\documentclass[12pt,final]{article}
\RequirePackage{tiaStyle}
\newcommand{\widecheck}{\check}

\title{An Efficient Approach for Computing Optimal Low-Rank Regularized Inverse Matrices}

\author{Julianne Chung  \\  Virginia Tech \\ Department of Mathematics\\ Blacksburg, VA, USA \\ \email{jmchung@vt.edu} \and Matthias Chung \\ Virginia Tech \\ Department of Mathematics\\ Blacksburg, VA, USA \\ \email{mcchung@vt.edu}\\ \\}
\begin{document}

\maketitle

\begin{abstract}
% Main Points:
% 	- new framework- don't need forward model, use training data
% 	- statistical framework (Bayes and empirical Bayes risk)
% 	- efficient rank-update approach to compute orim
% 	- predicting errors and uncertainties
% 
%
Standard regularization methods that are used to compute solutions to ill-posed inverse problems require knowledge of the forward model.  In many real-life applications, the forward model is not known, but training data is readily available. In this paper, we develop a new framework that uses training data, as a substitute for knowledge of the forward model, to compute an optimal low-rank regularized inverse matrix directly, allowing for very fast computation of a regularized solution.  We consider a statistical framework based on Bayes and empirical Bayes risk minimization to analyze theoretical properties of the problem.  We propose an efficient rank update approach for computing an optimal low-rank regularized inverse matrix for various error measures.  Numerical experiments demonstrate the benefits and potential applications of our approach to problems in signal and image processing.  
% An empirical Bayes risk minimization framework will be used to incorporate training data and to formulate the problem of computing an optimal regularized inverse matrix. Computing such a matrix is challenging, especially for large-scale problems, but sophisticated stochastic and numerical optimization methods that can impose constraints and take advantage of parallel computing will be investigated. Advanced tools will be developed for error prediction and uncertainty quantification, and real data from application scientists will be used to validate the methods developed in this project.
				
% 	Goal: Solve inverse problems where forward model is not known or only partially known
% 	Avoid including the forward model in the problem formulation, learns the information from provided training data
% 	Compute low-rank regularized inverse matrices that can then be used to solve inverse problems
% 	
% In this paper, we describe an efficient approach to compute low-rank regularized inverse approximations.  We propose a rank-$\ell$ update approach, which is motivated from theory developed for Bayes risk minimization.  For Bayes risk minimization, we show that a rank-$\ell$ update is optimal and describe some situations when numerical methods for computing the update may be required.  Then we consider the empirical Bayes risk minimization problem for different error measures, describe an efficient rank-$\ell$ update approach, and provide numerical results for problems from image processing.
\end{abstract}

% \pacs{00.00, 20.00, 42.10}
% \vspace{2pc}
% \submitto{\IP}

\section{Introduction} \label{sec:introduction} % (fold)
% \tia{
% - problem statement\\
% - Bayes + empirical risk\\
% - connection to previous work
% }
Inverse problems arise in scientific applications such as biomedical imaging, geophysics, computational biology, computer graphics, and security \cite{Engl2000,Vogel2002,Ramm2004,Tarantola2005,Hansen2006,Kirsch2011,Aster2012,Neto2012}.  Assume that $\bfA \in \bbR^{m \times n}$ and $\bfb \in \bbR^{m}$ are given, then the linear \emph{inverse problem} can be written as
\begin{equation}\label{eq:inverseproblem}
	\bfA \bfxi + \bfdelta = \bfb,
\end{equation}
where $\bfxi\in\bbR^n$ is the desired solution and $\bfdelta\in\bbR^m$ is additive noise.  
% Solving this inverse problem becomes challenging when assuming that the underlying model is ill-posed \cite{Tenorio2001}.  
We assume the underlying problem is \emph{ill-posed}.  A problem is ill-posed if the solution does not exist, is not unique, or does not depend continuously on the data \cite{Hadamard1923}.  A main challenge of ill-posed inverse problems is that small errors in the data may result in large errors in the reconstruction.  In order to obtain a more meaningful reconstruction, regularization in the form of prior knowledge is needed to stabilize the inversion process. 

Standard regularization methods require knowledge of $\mxA$ \cite{Hansen1998,Vogel2002}.  However, in many applications, the operator $\mxA$ may not be readily available, or is only partially known.  In this work, we investigate an approach that avoids including $\bfA$ in the problem formulation, and instead learns the necessary information from training or calibration data.  Assume training data $\bfb^{(k)}, \bfxi^{(k)}$ for $k = 1,\ldots, K$, for~\eqref{eq:inverseproblem} are given.  First, our goal is to find a rank-$r$ matrix $\bfZ \in \bbR^{n \times m}$ that gives a small error for the given training set.  That is, the \emph{sample mean error}, 
\begin{equation}
\label{eqn:samlemeanerror}
f_K(\bfZ) = \frac{1}{K} \sum_{k=1}^K e_k(\bfZ) =\frac{1}{K} \sum_{k=1}^K \rho \left(\bfZ \bfb^{(k)} - \bfxi^{(k)}\right)
\end{equation} 
should be small for some given error measure $\rho:\bbR^n \to \bbR^+_0$, e.g., $\rho(\bfx) =\norm[p]{\bfx}^p$. Here, $e_k(\bfZ)= \rho \left(\bfZ \bfb^{(k)} - \bfxi^{(k)}\right)$ refers to the \emph{sample error} for a single observation $k$. 
In this paper, we consider efficient methods for obtaining an \emph{optimal low-rank regularized inverse matrix},
\begin{equation}
	\label{eqn:empiricalmin}
	\widehat \bfZ = \argmin_{\rank{\bfZ} \leq r} \frac{1}{K} \sum_{k = 1}^K \rho\left(\bfZ \bfb^{(k)} - \bfxi^{(k)}\right)\,.
\end{equation}
Then, once matrix $\widehat \bfZ$ is computed, solving the inverse problem~\eqref{eq:inverseproblem} requires a simple matrix-vector multiplication: $\bfxi = \widehat \bfZ \bfb$. Matrix $\widehat \bfZ$ can be precomputed, so that regularized solutions to inverse problems can be obtained efficiently and accurately.

Methods that incorporate training data have been widely used in machine learning \cite{Mohri2012,AbuMostafa2012} and have been analyzed in the statistical learning literature \cite{Vapnik1998}.  
 % Methods from stochastic optimization \cite{}.
% We will show in Section \ref{sec:background} that problem \eqref{eqn:empiricalmin} can been considered an empirical Bayes risk minimization problem.
Most previous work on using training data for solving inverse problems has been restricted to minimizing the predictive mean squared error and require knowledge of the forward model \cite{Chung2011,Chung2012}.  In this paper, we provide a novel framework to solve inverse problems that does not require knowledge of the forward model.  We develop efficient methods to compute an optimal low-rank regularized inverse matrix for a variety of error measures.  In addition, we provide a natural framework for evaluating solutions and predicting uncertainties.

The paper is structured as follows.  In Section~\ref{sec:background} we provide background on regularization of inverse problems and discuss the connection to Bayes risk minimization. Some theoretical results are provided for the corresponding Bayes risk minimization problem.  In Section~\ref{sec:rank} we describe numerical methods to efficiently solve~\eqref{eqn:empiricalmin}, including a rank-update approach that can be used for various error measures and large-scale problems.  Numerical results are presented in Section~\ref{sec:numerical_results} and conclusions are provided in Section~\ref{sec:conclusions}.

% section introduction (end)
%!TEX root = orim.tex
\section{Background and Motivation} \label{sec:background} % (fold)
% \tia{Background on optimal low rank regularized inverse matrice
% - Inverse problems, regularization, TSVD
% - Bayes risk and epirical Bayes risk (theory for LR problems)
% mention special case 2norm}

The singular value decomposition (SVD) is an important tool for analyzing regularization methods. For any real $m \times n$ matrix $\mxA$, let $\mxA = \mxU \bfSigma \mxV\t$ be the SVD, where $\bfSigma$ is a diagonal matrix containing the singular values, $\sigma_1 \geq \sigma_2 \geq ...\geq \sigma_{\rank{\bfA}} > 0,$ and orthogonal matrices $\mxU$ and $\mxV$ contain the left and right singular vectors $\bfu_i\,, i=1,2,...,m,$ and $\bfv_i,\, i=1,2,...,n$ respectively \cite{Golub1996}.  For a given positive integer $r\leq \mbox{rank}(\mxA)$, define $$\bfA_r = \bfU_r \bfSigma_r \bfV_r\t,$$ where $\mxU_r = [\bfu_1, \bfu_2, ... \bfu_r] \in \bbR^{m \times r},\, \mxV_r = [\bfv_1, \bfv_2, ... \bfv_r] \in \bbR^{n \times r}$, and 
$$\bfSigma_r = 
\begin{bmatrix}
\sigma_1 & & & \\
& \sigma_2&&\\
&&\ddots& \\
&&& \sigma_r
\end{bmatrix} \in \bbR^{r \times r}.$$  
Then, we denote $\bfA_r^\dagger = \bfV_r \bfSigma_r^{-1} \bfU_r\t$ to be the pseudoinverse of $\bfA_r$ \cite{Stewart1998}.

For ill-posed inverse problems, the singular values of $\bfA$ decay to and cluster at zero.  Spectral filtering methods provide solutions where components of the solution corresponding to the small singular values are damped or neglected. For example, the rank-$r$ truncated SVD (TSVD) solution can be written as
\begin{equation}
	\label{eqn:TSVD}
	\bfxi_{\rm TSVD}  = \bfA_r^\dagger \bfb\,.
\end{equation} 
In Figure~\ref{fig:tsvd2}, we provide TSVD sample mean errors for different values of $r$ for a image deblurring problem.  For TSVD, the choice of rank is crucial. For small rank, there are not enough singular vectors to capture the details of the solution, but for large rank, small singular values are included, and noise corrupts the solution. This is an illustrative example, so details are not specified here.  See Section~\ref{sec:numerical_results} for more thorough numerical investigations.  
	\begin{figure}[tbp]
		\begin{center}
		\includegraphics[width = \textwidth]{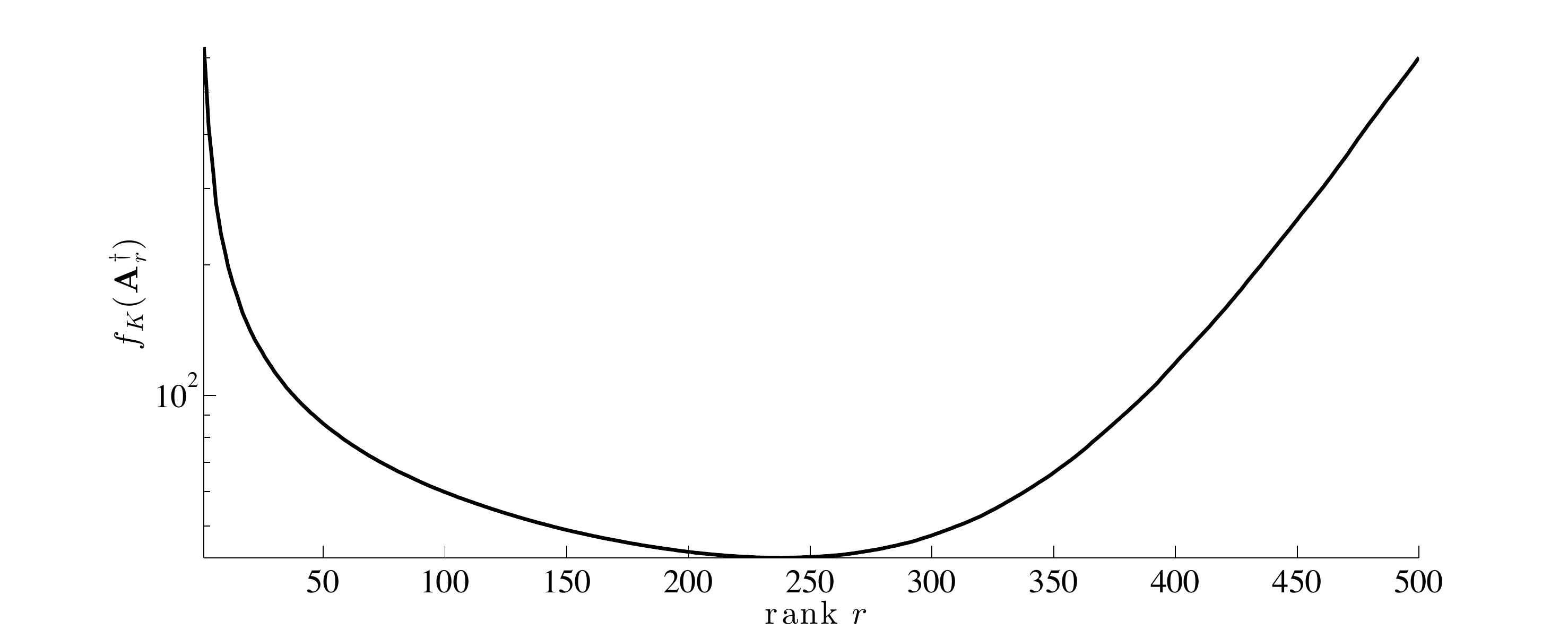}
		\end{center}
		\caption{Illustration of a typical sample mean error $f_K(\bfA^\dagger_r)$ on a logarithmic scale for Truncated SVD reconstructions for various ranks $r$.}
		\label{fig:tsvd2}
		\end{figure}

It is worth noting that the TSVD reconstruction matrix $\bfA_r^\dagger$ is \emph{one} example of a rank-$r$ matrix that solves~\eqref{eq:inverseproblem}.  In this paper, we seek an \emph{optimal} rank-$r$ matrix for all training data.  Computing a solution to problem~\eqref{eqn:empiricalmin} is non-trivial, and in the next section we provide some underlying theory from a Bayes risk perspective.

\subsection{Bayes risk minimization} % (fold)
\label{sub:bayes_risk_minimization}
% Suppose we have a set of possible signals $\bfXi \subseteq \bbR^{n}$ and a set of possible noise samples $\bfDelta \in \bbR^{m}$.  Assume that the training data is given by $\bfb^{(k)} = \mxA \bfxi^{(k)}+ \bfdelta^{(k)}$, where $\bfxi^{(k)} \in \bfXi$ was independently selected according to probability distribution $\calP_\xi$ and $\bfdelta^{(k)} \in \bfDelta$ was independently selected according to probability distribution $\calP_\delta$. 
% Then optimization problem~\eqref{eqn:empiricalmin} can be interpreted as an empirical Bayes risk minimization problem.
By interpreting problem~\eqref{eqn:empiricalmin} as an empirical Bayes risk minimization problem \cite{Carlin2000,Vapnik1998, Shapiro2009}, we can provide a statistical framework for analyzing the problem and motivate numerical algorithms for computing an optimal low-rank regularized inverse matrix.
In the following we assume that $\bfxi$ and $\bfdelta$ are random variables with given probability distributions $\calP_{\bfxi}$ and $\calP_{\bfdelta}$. Using the fact that $\bfZ \bfb -\bfxi = (\bfZ \bfA - \bfI)\bfxi + \bfZ \bfdelta$,
the \emph{Bayes risk minimization problem} can be written as,
\begin{equation}
	\label{eqn:Bayesmin}
\widecheck \bfZ =	\argmin_{\rank{\bfZ} \leq r} \ \bbE_{\bfxi,\bfdelta} \left(\rho((\bfZ \bfA - \bfI) \bfxi + \bfZ \bfdelta) \right) \, ,
\end{equation}
where $\bbE_{\bfxi,\bfdelta}$ is the expected value with respect to the joint distribution of $\bfxi$ and $\bfdelta$ \cite{Carlin2000,Vapnik1998}. The particular choices of $\rho,$ $\calP_{\bfxi},$ and $\calP_{\bfdelta}$ in~\eqref{eqn:Bayesmin} determine an optimal reconstruction matrix $\widecheck\bfZ$. Optimization problem~\eqref{eqn:empiricalmin} can be interpreted as an empirical Bayes risk minimization problem corresponding to Bayes risk minimization problem~\eqref{eqn:Bayesmin}, where $\bfxi^{(k)}$ and $\bfb^{(k)} = \bfA \bfxi^{(k)}+ \bfdelta^{(k)}$ denote realizations of the random variables	 $\bfxi$ and $\bfdelta$.

\subsection{Bayes Risk Minimization for the 2-norm case} % (fold)
\label{sub:Bayes_2_norm_case}
For the special case where $\rho(\bfx) =\norm[2]{\bfx}^2,$ the \emph{Bayes risk} can be written as 
\begin{equation}\label{eq:bayesRisk}
	\tilde f(\bfZ) = \bbE_{\bfxi,\bfdelta} \left(\norm[2]{(\bfZ \bfA - \bfI_n) \bfxi + \bfZ \bfdelta}^2 \right).
\end{equation} 
In this section, we provide closed form solutions for both the unconstrained and rank-constrained Bayes risk minimization problems for the 2-norm case.

% Without loss of generality, we assume distribution $\calP_{\bfDelta}$ has zero mean and the covariance matrix is a multiple of the identity matrix, $\bfC_{\bfdelta}=\alpha \mxI$.  For distribution $\calP_{\bfXi}$, we assume zero mean and positive definite covariance matrix, $\bfC_{\bfxi} = \bfM \bfM\t$.
We assume that the random variables $\bfxi$ and $\bfdelta$ are statistically independent.  
Let $\bfmu_{\bfxi}$ be the mean and $\bfC_{\bfxi}$ be the covariance matrix
for the distribution $\calP_{\bfxi}$.  Similarly,  $\bfmu_{\bfdelta}$
and $\bfC_{\bfdelta}$ are the mean and covariance matrix for $\calP_{\bfdelta}$.
We assume that both matrices are positive definite and that all entries
are finite.
One fact that will be useful to us is that the expectation of the quadratic form $\bfphi\t \bfK \bfphi$ involving a symmetric matrix $\bfK$ and a random variable $\bfphi$ with covariance matrix $\bfC_{\bfphi}$ can be written as
% One fact will be useful to us:
% the expectation of the quadratic form $\bfphi\t \bfK \bfphi$
% involving a random variable $\bfphi$ drawn from a set
% $\bfPhi \subseteq \bbR^{n}$ with distribution $\calP_{\bfphi}$
% and covariance matrix $\bfC_{\bfphi}$ can be written as \cite{Seber2003}
\[
\bbE_{\bfphi}(\bfphi\t \bfK \bfphi) = \bbE_{\bfphi}(\bfphi)\t \bfK \ \bbE_{\bfphi}(\bfphi) + \trace{\bfK \bfC_{\bfphi}},
\]
where $\trace{\,\cdot\,}$ denotes the trace of a matrix \cite{Seber2003}.
With this framework and the Cholesky decompositions
$\bfC_{\bfxi} = \bfM_{\bfxi} \bfM_{\bfxi}\t$  
and $\bfC_{\bfdelta} = \bfM_{\bfdelta} \bfM_{\bfdelta}\t$,
we can write the Bayes risk~\eqref{eq:bayesRisk} as,
\begin{align*}
	\tilde f(\bfZ) &= \bbE_{\bfxi,\bfdelta} \left(\norm[2]{(\bfZ \bfA - \bfI_n) \bfxi + \bfZ \bfdelta}^2 \right) \\
&= \bbE_{\bfxi} \left( \bfxi\t (\bfZ\bfA-\bfI_n)\t (\bfZ\bfA-\bfI_n)\bfxi \right) + 2 \ \bbE_{\bfdelta}( \bfdelta\t) \ \bfZ\t(\bfZ\bfA-\bfI_n) \ \bbE_{\bfxi}\left(\bfxi\right) + \bbE_{\bfdelta} \left( \bfdelta\t \bfZ\t\bfZ \bfdelta\right) \\
&= \bfmu_{\bfxi}\t (\bfZ\bfA-\bfI_n)\t (\bfZ\bfA-\bfI_n) \bfmu_{\bfxi} + \trace{(\bfZ\bfA-\bfI_n)\t (\bfZ\bfA-\bfI_n) \bfM_{\bfxi} \bfM_{\bfxi}\t}\\
&  + 2 \bfmu_{\bfdelta}\t  \bfZ\t(\bfZ\bfA-\bfI_n)  \bfmu_{\bfxi}
 + \bfmu_{\bfdelta}\t \bfZ\t\bfZ \bfmu_{\bfdelta} + \trace{\bfZ\t \bfZ \bfM_{\bfdelta} \bfM_{\bfdelta}\t}\\ 
&=  \norm[2]{(\bfZ\bfA-\bfI_n) \bfmu_{\bfxi}}^2 + \norm[\rm F]{(\bfZ\bfA-\bfI_n)\bfM_{\bfxi}}^2 + 2 \bfmu_{\bfdelta}\t  \bfZ\t(\bfZ\bfA-\bfI_n)  \bfmu_{\bfxi}
 + \norm[2]{\bfZ \bfmu_{\bfdelta}}^2 + \norm[\rm F]{\bfZ \bfM_{\bfdelta}}^2,
\end{align*}
% \begin{eqnarray*}
% 	\tilde f(\bfZ) &=& \bbE_{\bfxi,\bfdelta} \left(\norm[2]{(\bfZ \bfA - \bfI_n) \bfxi + \bfZ \bfdelta}^2 \right) \\
% &=& \bbE_{\bfxi} \left( \bfxi\t (\bfZ\bfA-\bfI_n)\t (\bfZ\bfA-\bfI_n)\bfxi \right) + 2 \ \bbE_{\bfdelta}( \bfdelta\t) \ \bfZ\t(\bfZ\bfA-\bfI_n) \ \bbE_{\bfxi}\left(\bfxi\right) + \bbE_{\bfdelta} \left( \bfdelta\t \bfZ\t\bfZ \bfdelta\right) \\
% &=& \bfmu_{\bfxi}\t (\bfZ\bfA-\bfI_n)\t (\bfZ\bfA-\bfI_n) \bfmu_{\bfxi} + \trace{(\bfZ\bfA-\bfI_n)\t (\bfZ\bfA-\bfI_n) \bfM_{\bfxi} \bfM_{\bfxi}\t}\\
% & & + 2 \bfmu_{\bfdelta}\t  \bfZ\t(\bfZ\bfA-\bfI_n)  \bfmu_{\bfxi}
%  + \bfmu_{\bfdelta}\t \bfZ\t\bfZ \bfmu_{\bfdelta} + \trace{\bfZ\t \bfZ \bfM_{\bfdelta} \bfM_{\bfdelta}\t}\\ 
% &= & \norm[2]{(\bfZ\bfA-\bfI_n) \bfmu_{\bfxi}}^2 + \norm[\rm F]{(\bfZ\bfA-\bfI_n)\bfM_{\bfxi}}^2 \\
% &&+ 2 \bfmu_{\bfdelta}\t  \bfZ\t(\bfZ\bfA-\bfI_n)  \bfmu_{\bfxi}
%  + \norm[2]{\bfZ \bfmu_{\bfdelta}}^2 + \norm[\rm F]{\bfZ \bfM_{\bfdelta}}^2,
% \end{eqnarray*}
where $\norm[\rm F]{\, \cdot \,}$ denotes the Frobenius norm of a matrix.

We can simplify this final expression 
by {\em pre-whitening} the noise $\bfdelta$, thereby converting to a coordinate
system where the noise mean is zero and the
noise covariance is a multiple of the identity.
Define the random variable 

\begin{equation}
	\bfchi = \bfM_{\bfdelta}^{-1}(\bfdelta - \bfmu_{\bfdelta})\,.
\end{equation}
Then, $\bfmu_{\bfchi} = \bfzero$ and $\bfC_{\bfchi} = \mxI_m,$ and we have
\begin{align}
	\bfM_{\bfdelta}^{-1} \bfA \bfxi + \bfchi & = \bfM_{\bfdelta}^{-1}(\bfb-\bfmu_{\bfdelta})\,,
\end{align}
which can be written as,
\begin{align}
	\widetilde \bfA \bfxi + \bfchi &= \widetilde \bfb\,,
\end{align}
where $\widetilde \bfA = \bfM_{\bfdelta}^{-1} \bfA$ and 
$\widetilde \bfb = \bfM_{\bfdelta}^{-1}(\bfb-\bfmu_{\bfdelta}).$  
Thus, without loss of generality
(although with the need for numerical care if $\bfM_{\bfdelta}$ is 
ill-conditioned), we can always assume that $\bfmu_{\delta} = \bfzero$ and $\bfC_{\bfdelta} = \eta^2 \mxI_m$.
Therefore, our Bayes risk (\ref{eq:bayesRisk}) takes the form
\begin{equation}
\tilde f(\bfZ) = \norm[2]{(\bfZ\bfA-\bfI_n) \bfmu_{\bfxi}}^2 + \norm[\rm F]{(\bfZ\bfA-\bfI_n)\bfM_{\bfxi}}^2
 + \eta^2\norm[\rm F]{\bfZ}^2. \label{eq:Bayesrisk}
\end{equation} 

Under certain conditions, summarized in the following
theorem, 
there exists a unique global minimizer for (\ref{eq:Bayesrisk}).

\begin{theorem}\label{thm:genBays}
Given $\bfA \in \bbR^{m \times n}$, $\bfM_{\bfxi} \in \bbR^{n \times n}$, $\bfmu_{\bfxi} \in \bbR^n$ and $\eta \in \bbR$, 
consider the problem of finding a matrix $\widecheck \bfZ \in \bbR^{n \times m}$ with
\begin{equation}
	\widecheck \bfZ = \argmin_{\bfZ} % \tilde f(\bfZ) = 
\norm[2]{(\bfZ\bfA-\bfI_n) \bfmu_{\bfxi}}^2 + \norm[\rm F]{(\bfZ\bfA-\bfI_n)\bfM_{\bfxi}}^2 + \eta^2\norm[\rm F]{\bfZ}^2.
\end{equation}
If either $\eta \ne 0$ or if $\bfA$ has rank $m$ and
$\bfmu_{\bfxi} \bfmu_{\bfxi}\t + \bfM_{\bfxi} \bfM_{\bfxi}\t$ is nonsingular,
then the unique global minimizer is
\begin{equation}
	\widecheck \bfZ = (\bfmu_{\bfxi} \bfmu_{\bfxi}\t + \bfM_{\bfxi} \bfM_{\bfxi}\t) \bfA\t [\bfA (\bfmu_{\bfxi} \bfmu_{\bfxi}\t + \bfM_{\bfxi} \bfM_{\bfxi}\t) \bfA\t  +   \eta^2 \bfI_m  ]^{-1}.
\end{equation}
\end{theorem}

\begin{proof}
	Using derivatives of norms and traces, we get
	\begin{equation}
		\frac{\d \tilde f(\bfZ)}{\d\bfZ} = 2 (\bfZ \bfA - \bfI_n) \bfmu_{\bfxi} \bfmu_{\bfxi}\t \bfA\t + 2 (\bfZ \bfA - \bfI_n) \bfM_{\bfxi} \bfM_{\bfxi}\t \bfA\t +  2 \eta^2 \bfZ.
	\end{equation}
	Setting the above derivative equal to zero gives the following first order condition,
	\begin{equation}
		 \bfZ [\bfA (\bfmu_{\bfxi} \bfmu_{\bfxi}\t + \bfM_{\bfxi} \bfM_{\bfxi}\t) \bfA\t  +  \eta^2 \bfI_m  ]= (\bfmu_{\bfxi} \bfmu_{\bfxi}\t + \bfM_{\bfxi} \bfM_{\bfxi}\t) \bfA\t .
	\end{equation}
	Under our assumptions, the matrix $\bfA (\bfmu_{\bfxi} \bfmu_{\bfxi}\t + \bfM_{\bfxi} \bfM_{\bfxi}\t) \bfA\t  +  \eta^2 \bfI_m$ is invertible, 
so we get the unique solution $\widecheck \bfZ$ that is given in the statement of the Theorem.
It is a minimizer since the second derivative matrix of $\tilde f (\bfZ)$ is positive definite.
\end{proof}

To summarize, we have minimized the Bayes risk (\ref{eq:bayesRisk})
under the assumptions that $\bfxi$ and $\bfdelta$ are statistically independent random variables with symmetric positive definite finite covariance matrices.

Next we consider the rank-constrained Bayes risk minimization problem.
With the additional assumption that $\bfmu_{\bfxi} = \bfzero$, problem~\eqref{eqn:Bayesmin} can be written as
\begin{equation}
	\label{eqn:thm1}
		\min_{\rank{\bfZ} \leq r} \ \norm[\rm F]{(\bfZ \bfA - \bfI)\bfM}^2 + \alpha^2 \norm[\rm F]{\bfZ}^2,
\end{equation}
where $\bfM = \bfM_{\bfxi}$ and $\alpha = \eta.$
A closed form solution to~\eqref{eqn:thm1} was recently provided in~\cite{Chung2014b}, and the result is summarized in the following theorem.
\begin{theorem} \label{thm:general}
Given a matrix $\bfA \in \bbR^{m\times n}$ where $\rank{\bfA} \le n \le m$
and an invertible matrix $\bfM \in \bbR^{n \times n}$,
let their generalized singular value decomposition be
$\bfA = \bfU \bfSigma \bfG^{-1}$, $\bfM = \bfG \bfS^{-1} \bfV\t$.
Let $\alpha$ be a given parameter, nonzero if $\rank{\bfA} < m$.
Let $r \le \rank{\bfA}$ be a given positive integer.
Define $\bfD = \bfSigma \bfS^{-2} \bfSigma\t + \alpha^2 \bfI$.
Let the symmetric matrix
$\bfH = \bfG \bfS^{-4} \bfSigma\t \bfD^{-1} \bfSigma \bfG\t$
have eigenvalue decomposition 
$\bfH = \widecheck \bfV \bfLambda \widecheck \bfV\t$
with eigenvalues ordered so that 
$\lambda_j \geq \lambda_i$ for $j < i \le n$.
Then a global minimizer $\widecheck \bfZ \in \bbR^{n \times m}$ of problem \eqref{eqn:thm1}
is
\begin{equation} \label{eq:zhat}
	\widecheck \bfZ = \widecheck \bfV_r \widecheck \bfV_r\t \bfG \bfS^{-2}
	       \bfSigma\t \bfD^{-1} \bfU\t,
\end{equation}
where $\widecheck \bfV_r$ contains the first $r$ columns of $\widecheck \bfV$.
Moreover this $\widecheck \bfZ$ is the {\em unique} global minimizer of~\eqref{eqn:thm1} if and only if $\lambda_r > \lambda_{r+1}$.
\end{theorem}

Next we consider the case where $\bfC_{\bfxi} = \beta^2 \bfI_n$ for some scalar $\beta$.  In this case the Bayes risk simplifies to \begin{equation}
\tilde f(\bfZ)= \beta^2 \norm[\rm F]{\bfZ\bfA - \bfI_n}^2 + \eta^2 \norm[\rm F]{\bfZ}^2, \label{eq:bayesRisk1}
\end{equation}
and, by rescaling, the low-rank Bayes risk minimization problem can be restated as optimization problem~\eqref{eqn:thm1}
where $\alpha^2 = \frac{\eta^2}{\beta^2}$ represents the noise-to-signal ratio
and $\bfM = \bfI_n$.  It was shown in \cite{Chung2014b} that for this problem, the global minimizer simplifies to  
\begin{equation}
	\widecheck \bfZ = \bfV_r \bfPsi_r \bfU_r\t ,
\end{equation}
where $\bfPsi_r = \diag{\frac{\sigma_1}{\sigma_1^2+ \alpha^2}, \ldots, \frac{\sigma_J}{\sigma_r^2+ \alpha^2}}$ and $\bfV_r$ and $\bfU_r$ were defined in the beginning of this section. Moreover, $\widecheck \bfZ$ is the {\em unique} global minimizer if and only if $\sigma_r > \sigma_{r+1}$.  For problems where $\bfA$ is known, matrix $\widecheck \bfZ$ can be used to compute a filtered solution to inverse problems.  That is, 
$$ \bfx_{\rm filter} \equiv \widecheck \bfZ \bfb = \sum_{j=1}^n \phi_j \frac{\bfu_j\t \bfb}{\sigma_j}\bfv_j, $$
where the filter factors are given by
\begin{equation*}
	\phi_j = \left\{\begin{array}{cl}\frac{\sigma_j^2}{\sigma_j^2 + \alpha^2}, & \mbox{ for } j \leq r, \\ 0, & \mbox{ for }j>r. \end{array}\right. 
\end{equation*}
Here, $\widecheck \bfZ$ is the \emph{truncated Tikhonov} reconstruction matrix \cite{Chung2014b}. 

% subsection bayes_risk_minimization (end)
% section background (end)

\section{Computing Low-rank Regularized Inverse Matrices} \label{sec:rank} % (fold)
In this section, we describe numerical methods to efficiently compute a low-rank regularized inverse matrix, i.e., a solution to \eqref{eqn:empiricalmin}, for problems where matrix $\bfA$ is unknown and training data are available.  We first describe a method that can be used for the special case where $\rho = \norm[2]{\, \cdot \,}^2$.  Further we discuss a rank-update approach that can be utilized for large-scale problems and for general error measures.

\subsection{2-norm case} % (fold)
\label{sub:2_norm_case}
For the case $\rho = \norm[2]{\, \cdot \,}^2$, we  can use the relationship between the 2-norm of a vector and the Frobenius norm of a matrix to rewrite optimization problem~\eqref{eqn:empiricalmin} as
\begin{equation}\label{eq:min2norm}
 	\widehat \bfZ = \argmin_{\rank{\bfZ}\leq r} \frac{1}{K}\sum_{k=1}^{K} \norm[2]{\bfZ\bfb^{(k)}-\bfxi^{(k)}}^2 = \frac{1}{K}\norm[\rm F]{\bfZ\bfB -\bfC}^2\,,
 \end{equation} 
where $\bfB = [\bfb^{(1)},\ldots, \bfb^{(K)}]$ and $\bfC = [\bfxi^{(1)},\ldots, \bfxi^{(K)}]$. A closed form solution for~\eqref{eq:min2norm} exists, and the result is summarized in the following theorem.
% A close form solution for~\eqref{eq:min2norm} exist and is given by the following theorem.
% Let $\bfP = \bfC \tilde \bfV_b\tilde \bfV_b\t$ where rank $b = rank(\bfB)$ and $\bfV$ the matrix of the right singular vectors of $\bfB$. Let further 

% Let $\bfB^\dagger$ be the pseudo-inverse of $\bfB$ and let $\bfU_\bfB \bfSigma_\bfB \bfV_\bfB\t$ be the singular value decomposition of $\bfB$ with $\bfSigma_\bfB = \diag{\sigma_1,\ldots,\sigma_{\min(m,K)}}$ where $\sigma_1 \geq \ldots \geq \sigma_{\min(m,K)}$ and with $\bfU_\bfB = [\bfu_1,\ldots,\bfu_m]$ and $\bfV_\bfB = [\bfv_1,\ldots,\bfv_K]$. Further let $\bfP = \sum_{i= 1}^{\rank{\bfB}}\bfv_i \bfv_i\t$ and let denote the $\bfB_J = \sum_{j = 1}^J \sigma_j \bfu_j \bfv_j\t$ for $J< \rank{\bfB}$ and $\bfB_J = \bfB$ otherwise.
\begin{theorem}
	\label{thm:Sondermann}
	Let matrices $\bfB\in\bbR^{m\times K}$ and $\bfC\in\bbR^{n\times K}$ be given. Further let $\tilde \bfV$ be the matrix of right singular vectors of $\bfB$ and define $\bfP = \bfC \tilde \bfV_s (\tilde \bfV_s)\t$ where $s = \rank{\bfB}$. Then
	$$ \widehat\bfZ = \bfP_r \bfB^{\dagger}$$
	is a solution to the minimization problem~\eqref{eq:min2norm}. This solution is unique if and only if either $r \geq \rank{\bfP}$ or 
	$ 1\leq r \leq \rank{\bfP}$ and $\bar \sigma_r > \bar \sigma_{r+1}$, where $\bar \sigma_r$ and $\bar \sigma_{r+1}$  denote the $r$ and $(r+1)$st singular values of $\bfP$.
\end{theorem}
\begin{proof}
	This result is a special case of the main result of Sondermann \cite{Sondermann1986} and was reproved with if-and-only-if uniqueness conditions in Theorem 2.1 of Friedland and Torokhti \cite{Friedland2007}.
\end{proof}

The main computational cost of constructing $\widehat \bfZ$ using Theorem \ref{thm:Sondermann} lies in two SVD computations (for matrices $\bfB$ and $\bfP$).  If these computations are possible and minimizing the mean squared error is desired, then we propose to use the above result to compute optimal low-rank regularized inverse matrix $\widehat \bfZ$, which in turn can be used to solve inverse problems \cite{Chung2013a}.  
% subsection 2_norm_case (end)
% Previous work - If Z has structure - optimal filter
% In the following we want to assume for convenience that a unique global minimizer for Corollary~\ref{coro:thm} exist.  

\subsection{A rank-update approach} % (fold)
\label{sec:rank_update}
For general error measures or for large problems with large training sets, we propose a rank-update approach for obtaining optimal low-rank regularized inverse matrix $\widehat \bfZ$.  The approach is motivated by the following corollary to Theorem \ref{thm:general}.
\begin{corollary}\label{coro:thm}
Assume all conditions of Theorem~\ref{thm:general} are fulfilled. Let $\widecheck \bfZ^J$ be a global minimizer of~\eqref{eqn:thm1} for maximal rank $J$ and $\widecheck \bfZ^{J+\ell}$ be a global minimizer of~\eqref{eqn:thm1} for maximal rank $J+\ell$. Then $\bar \bfZ^{\ell} = \widecheck \bfZ^{J+\ell}-\widecheck \bfZ^J$ is the global minimizer of 
	\begin{equation}\label{eq:updateFormula}
		\min_{\rank{\bfZ}\leq \ell}\norm[\rm F]{((\widecheck\bfZ^J+\bfZ)\bfA-\bfI)\bfM}^2 + \alpha^2\norm[\rm F]{\widecheck\bfZ^J+\bfZ}^2.
	\end{equation}
	Furthermore, $\bar\bfZ^{\ell}$ is the unique global minimizer if and only if $\lambda_J > \lambda_{J+1}$ and $\lambda_{J+\ell} > \lambda_{J+\ell+1}$ with $J+\ell+1<\rank{\bfA}$.
\end{corollary}

Corollary \ref{coro:thm} states that in the Bayes framework for the 2-norm case, a rank-update approach (in exact arithmetic) can be used to compute an optimal low-rank regularized matrix. That is, assume we are given the rank-$J$ approximation $\widecheck\bfZ^J$ and we are interested in finding $\widecheck \bfZ^{J+\ell}$. To calculate the optimal rank-$(J+\ell)$ approximation $\widecheck \bfZ^{J+\ell}$, we just need to solve a rank-$\ell$ optimization problem of the form~\eqref{eq:updateFormula} and then update the solution, $\widecheck\bfZ^{J+\ell} = \widecheck \bfZ^J + \bar\bfZ^{\ell}$. 

% , but can be generalized for different choices of $\rho(\bfx)$ and to problems that do not assume an underlying probability distribution for the data. 
% Recall that the goal is to compute a solution to the rank-constrained empirical Bayes minimization problem~\eqref{eqn:empiricalmin}.  
We are interested in the empirical Bayes problem, and we extend the rank-update approach to problem~\eqref{eqn:empiricalmin}, where the basic idea is to iteratively build up matrices $\widehat \bfX \in \bbR^{n \times r}$ and $\widehat \bfY \in \bbR^{m \times r}$ such that $\widehat \bfZ = \widehat \bfX \widehat \bfY\t.$  Suppose we are given the optimal rank-$J$ matrix $\widehat \bfZ$. That is, $\widehat \bfZ$ solves~\eqref{eqn:empiricalmin} for $r=J$.  Then, we compute a rank-$\ell$ update to $\widehat \bfZ$ by solving the following rank-$\ell$ optimization problem,
\begin{equation}
\label{eqn:updateproblem}
	(\widehat \bfX^\ell, \widehat \bfY^\ell) = \argmin_{\bfX \in \bbR^{n \times \ell}, \bfY^{m \times \ell}} f_K(\bfX,\bfY) = \frac{1}{K} \sum_{k = 1}^K \rho((\widehat \bfZ+\bfX\bfY\t) \bfb^{(k)} - \bfxi^{(k)}) \,,
\end{equation}	
and a rank-$(J+\ell)$ regularized inverse matrix can be computed as $\widehat \bfZ = \widehat \bfZ+\widehat \bfX^\ell (\widehat \bfY^\ell)\t$.  This process continues until the desired rank is achieved.  A summary of the general algorithm can be found below, see Algorithm~\ref{alg:update}.

\begin{algorithm} 
\caption{Rank-$\ell$ update approach for computing a solution to~\eqref{eqn:empiricalmin}}\label{alg:update}
\begin{algorithmic}[1]
\STATE Initialize $\widehat \bfZ = \bfzero_{n \times m}$, $\widehat \bfX = [\,\,]$, $\widehat \bfY = [\,\,]$
% \WHILE{$\rank {\widehat\bfZ} \leq r$}
\WHILE{stopping criteria not satisfied}
\STATE solve rank-$\ell$ update problem~\eqref{eqn:updateproblem}
\STATE update matrix inverse approximation: $\widehat \bfZ \longleftarrow \widehat \bfZ+\widehat \bfX^\ell (\widehat \bfY^\ell)\t $
\STATE update solutions: $\widehat \bfX \longleftarrow [\widehat \bfX, \widehat \bfX^\ell],  \widehat \bfY \longleftarrow [\widehat \bfY, \widehat \bfY^\ell]$
\ENDWHILE
\end{algorithmic}
\end{algorithm}
% {\small
% \begin{center}
% \begin{tabular}{|c|}\hline \\[-6pt]
% {\bf Rank-$\ell$ update approach for computing a solution to \eqref{eqn:empiricalmin}} \\[6pt] \hline \\
% [-6pt]
% \begin{minipage}[c]{4.5in}
% \begin{tabular}{lcl}
% \multicolumn{3}{l}{Initialize $\widehat \bfZ = \bfzero_{n \times m}$, $\widehat \bfX = \{ \}$, $\widehat \bfY = \{ \}$} \\[3pt]
% \multicolumn{3}{l}{while $\rank {\widehat\bfZ} \leq r$} \\[3pt]
% & & $(1)$ solve rank-$\ell$ update problem~\eqref{eqn:updateproblem} \\ [3pt]
% & & $(2)$ update matrix inverse approximation: $\widehat \bfZ = \widehat \bfZ+\widehat \bfX^\ell (\widehat \bfY^\ell)\t $\\[3pt]
% & & $(3)$ update solutions: $\widehat \bfX = [\widehat \bfX, \widehat \bfX^\ell],  \widehat \bfY = [\widehat \bfY, \widehat \bfY^\ell]$\\[3pt]
% \multicolumn{3}{l}{end} \\
% \end{tabular}
% \end{minipage} \\ \\ \hline
% \end{tabular}
% \end{center}
% }
\noindent Next we provide some computational remarks for Algorithm~\ref{alg:update}.

\subsubsection*{Choice of $\ell$.} At each iteration of Algorithm~\ref{alg:update}, we must solve optimization problem~\eqref{eqn:updateproblem}, which has $(m+n)\,\ell$ unknowns. Hence, the choice of $\ell$ will depend on the size of the problem and may even vary per iteration, e.g., $\ell_1, \ell_2,\ldots$.  For large-scale problems such as those arising in image reconstruction, we recommend updates of size $\ell=1$ to keep optimization problem~\eqref{eqn:updateproblem} as small as possible.

\subsubsection*{Stopping criteria.}
Stopping criteria for Algorithm~\ref{alg:update} should have four components. First, a safe-guard maximal rank condition.  That is, we stop if $\rank{\widehat \bfZ}$ is greater than some pre-determined integer value $r_{\max}$. Secondly, let $f_{\max}$ be a user-defined error tolerance for the function value, then we stop if $f_K \leq f_{\max}$.
Similarly, we can monitor the relative improvement of $f_K$ in each iteration and stop if the relative improvement is smaller than some tolerance. Lastly, if a rank-$\ell$ update does not substantially improve $f_K$, the iteration should be terminated.  More specifically, let $\widehat \bfX^\ell = [\widehat\bfx_1,\ldots,\widehat\bfx_\ell]$ and choose tolerance $\tau$.  Then we can control the rank deficiency of $\bfX^\ell$ by stopping if $\min \{\norm[\infty]{\widehat\bfx_1},\ldots,\norm[\infty]{\widehat\bfx_\ell}\} \leq \tau$.  The particular choices of stopping criteria and tolerances are problem dependent.

\subsubsection*{Solving rank-$\ell$ update problem~\eqref{eqn:updateproblem}.} % (fold)
\label{sub:solving_the_rank_ell_}
The main challenge for Algorithm~\ref{alg:update} is in line~3, where we need to solve the rank-$\ell$ update problem~\eqref{eqn:updateproblem}. First, we remark that the solution $(\widehat \bfX^\ell, \widehat \bfY^\ell)$ is not unique.  In fact, $\widehat \bfX^\ell (\widehat \bfY^\ell)\t = \frac{1}{a}\widehat \bfX^\ell \,  (a\widehat \bfY^\ell)\t$ for any $a\neq 0$. As a consequence optimization problem~\eqref{eqn:updateproblem} is ill-posed and a particular solution needs to be selected. Various computational approaches exist to handle this kind of ill-posedness. One way to tackle this problem is to constrain each column in $\bfX$ to have norm $1$, i.e., $\norm[2]{\bfX_{(:,j)}} = 1$ for $j = 1,\ldots,\ell$, and solve the constrained or relaxed joint optimization problem. We found this approach to be computationally inefficient and instead propose an alternating optimization approach. Notice that for any convex function $\rho$, function $f_K(\bfX,\bfY)$ is convex in $\bfX$ and convex in $\bfY$, so an alternating optimization approach avoids the above ill-posedness.  Until convergence, the solution to $\min_{\bfY}  f_K(\bfX, \bfY)$ is computed for fixed $\bfX$, and then the solution to $\min_{\bfX}  f_K(\bfX, \bfY)$ is computed for fixed $\bfY$. Columns of $\bfX$ are normalized at each iteration.  The alternating optimization is summarized in Algorithm~\ref{alg:alter}. 

% The main challenge for Algorithm~\ref{alg:update} is line~3, where we need to solve the rank-$\ell$ update problem~\eqref{eqn:updateproblem}. First, we notice that a matrix $\bfZ$ does not have a unique rank-$\ell$ decomposition $\bfX\bfY\t$, in fact  $\bfZ = \bfX\bfY\t = \frac{1}{a}\bfX \, a \bfY\t$ for any $a\neq 0$. As a consequence optimization problem~\eqref{eqn:updateproblem} is ill-posed and a particular solution $(\widehat\bfX^\ell,\widehat\bfY^\ell)$ needs to be selected. Various computational approaches exist to handle this kind of ill-posedness. One way to tackle this problem is to select a specific solution by adding the constraints such as normalized columns in $\bfX$, i.e., $\norm[2]{\bfX_{(:,j)}} = 1$ for $j = 1,\ldots,\ell$ and solve the constrained or relaxed optimization problem. We found this approach to be computational not efficient and chose instead an alternating optimization approach. Notice that for any convex function $\rho$ the function $f_K(\bfX,\bfY)$ is convex in $\bfX$ and convex in $\bfY$. Hence an alternating optimization solving
% $\min_{\bfX}  f_K(\bfX, \bfY)$ with fixed $\bfY$ and $\min_{\bfY}  f_K(\bfX, \bfY)$ with fixed $\bfX$ avoids above ill-posedness. The alternating optimization is summarized in Algorithm~\ref{alg:alter}. 
% 
% - non uniqueness considerations

\begin{algorithm}

\caption{Alternating optimization for solving~\eqref{eqn:updateproblem}} \label{alg:alter}
\begin{algorithmic}[1]
% \REQUIRE 
\STATE choose $\bfX^{(0)} \in \bbR^{m\times \ell}$, $i = 0$
% \WHILE{still improving}
\WHILE{stopping criteria not satisfied}
\STATE Solve $\bfY^{(i+1)} = \argmin_{\bfY}  f_K(\bfX^{(i)},\bfY)$
\STATE Solve $\bfX^{(i+1)} = \argmin_{\bfX}  f_K(\bfX,\bfY^{(i+1)})$
\STATE $\bfX_{(:,j)}^{(i+1)} \longleftarrow \bfX_{(:,j)}^{(i+1)}/\norm[2]{\bfX_{(:,j)}^{(i+1)}}$ for $j = 1,\ldots,\ell$
\STATE $i \longleftarrow i+1$
\ENDWHILE
\end{algorithmic}
\end{algorithm}

Methods for solving optimization problems in lines~3 and~4 of Algorithm~\ref{alg:alter} depend on the specific form of $\rho$. For instance for $\rho = \norm[p]{\, \cdot \,}^p$ with $2 \leq p < \infty$ we suggest using Gauss-Newton type methods. 
% For error measures $\rho$ with discontinuous second derivatives such as the Huber function 
% \begin{equation}
% 	\rho(\bfx) = \sum_{j = 1}^n \psi(x_j)
% \end{equation}
% with
% \begin{equation}
% 	\psi(x) = 
% 		\begin{cases} 
% 		    \abs{x}-\frac{\eps}{2}, & \mbox{if } \abs{x}\geq\eps,\\
% 	        \frac{1}{2\eps}x^2,  &\mbox{if } \abs{x}<\eps,
% 		\end{cases} 
% \end{equation}
% for some $\eps>0$ 
For $p$-norms with $1<p<2$ we propose to use Quasi-Newton type methods such as SR1, BFGS, and LBFGS \cite{Nocedal1999} or utilize smoothing techniques to compensate for discontinuities in the derivative. Standard stopping criteria can be used for these optimization problems, as well as for Algorithm~\ref{alg:alter} \cite{Gill1981}.  The required gradients of $f_K$ and Gauss-Newton approximation to the Hessian are provided below. 

% First let $\bfe^{(k)} = \bfX\bfY\t\bfb^{(k)} - \bfc^{(k)}$ be the $k$-th residual and let $\bfrho^{(k)} = \nabla_{\bfe^{(k)}} \rho(\bfe^{(k)})$ denote the gradient and $\bfD^{(k)} = \nabla^2_{\bfe^{(k)}} \rho(\bfe^{(k)})$ the Hessian of $\rho$ in $\bfe^{(k)}$ for $k = 1,\ldots, f_K$. We further define the column vector $\bfrho = [\bfrho^{(1)};\ldots;\bfrho^{(K)}]$ and $\bfD = [\bfD^{(1)},\dots,\bfD^{(K)}]$.

Let $\bfe^{(k)} = \bfX\bfY\t\bfb^{(k)} - \bfc^{(k)}$, $\bfx = \vec{\bfX}$, and $\bfy = \vec{\bfY}$.  Then the gradients are given by
\begin{equation}
	\nabla_{\bfx}\, f_K  = \frac{1}{K} \left(\bfJ_{\bfx}\right)\t \, \bfrho \qquad \mbox{and} \qquad
	\nabla_{\bfy}\, f_K = \frac{1}{K} \left(\bfJ_{\bfy}\right)\t \, \bfrho,
\end{equation}
where $\bfrho = \vec{[\bfrho^{(1)},\ldots,\bfrho^{(K)}]}$ with $\bfrho^{(k)} = \nabla\rho$ evaluated at $\bfe^{(k)}$. The Jacobian matrices have the form,
$$ \bfJ_\bfx = \begin{bmatrix}
	\bfJ^{(1)}_{\bfx_1} & \cdots & \bfJ^{(1)}_{\bfx_\ell}\\
	  \vdots  &   & \vdots \\
	\bfJ^{(K)}_{\bfx_1} & \cdots & \bfJ^{(K)}_{\bfx_\ell}
\end{bmatrix} \qquad \mbox{and} \qquad
 \bfJ_\bfy = \begin{bmatrix}
	\bfJ^{(1)}_{\bfy_1} & \cdots & \bfJ^{(1)}_{\bfy_\ell}\\
	  \vdots  &   & \vdots \\
	\bfJ^{(K)}_{\bfy_1} & \cdots & \bfJ^{(K)}_{\bfy_\ell}
\end{bmatrix},
$$
where $ \bfJ^{(k)}_{\bfx_j} = \bfy_j\t \bfb^{(k)} \bfI_n$ and $\bfJ^{(k)}_{\bfy_j} = \bfx_j \left(\bfb^{(k)}\right)\t$. The Gauss-Newton approximations of the Hessians are given by
	\begin{equation}
	\bfH_\bfx = \frac{1}{K} \left(\bfJ_\bfx\right)\t \bfD \,\bfJ_\bfx \qquad \mbox{and} \qquad
	\bfH_\bfy = \frac{1}{K} \left(\bfJ_\bfy\right)\t \bfD \,\bfJ_\bfy, 
\end{equation}
where $\bfD = \diag{\bfD^{(1)},\ldots,\bfD^{(K)}}$ and $\bfD^{(k)} = \nabla^2 \rho$ evaluated at $\bfe^{(k)}$. Note that $\bfD$ is diagonal for any error measures $\rho$ of the form $\rho(\bfx) = \sum_{i = 1}^n \phi(x_i)$ which includes the Huber function and all $p$-norms with $p<\infty$.
% section general_measures (end)
% subsection solving_the_rank_ell_ (end)

% \tia{How to select initial guess, $\bfX^{(0)} \in \bbR^{m\times \ell}$?}

Another important remark is that for computational efficiency, we do not need to form $\widehat \bfZ$ explicitly.  That is, if we let $\bfc^{(k)} = \bfxi^{(k)} - \widehat \bfZ \bfb^{(k)},$ then~\eqref{eqn:updateproblem} can be reformulated as
	\begin{equation}
		\label{eq:EmpRiskLowRank} 
	(\widehat \bfX^\ell, \widehat \bfY^\ell) = \argmin_{\bfX \in \bbR^{n \times \ell}, \bfY^{m \times \ell}}  \frac{1}{K}\sum_{k=1}^K \rho(\bfX \bfY\t \bfb\pow{k} - \bfc\pow{k}),
	\end{equation}
	where vectors $\bfc^{(k)}$ should be updated at each iteration as $\bfc^{(k)} \longleftarrow \bfc^{(k)} - \widehat \bfX ^\ell (\widehat \bfY^\ell)\t \bfb^{(k)}.$

% Algorithm~\ref{alg:alter} requires stopping criteria for the alternating direction loop. Here, we choose typical stopping criteria from optimization \cite{Gill????}
% 
% 
%   STOP1 = abs(fOld - f) <= tol * (1 + abs(f));
%   STOP2 = norm(xOld - x,'inf') <= sqrt(tol) * (1 + norm(x,'inf'));
%   STOP3 = norm(yOld - y,'inf') <= sqrt(tol) * (1 + norm(y,'inf'));

% A variety of stopping criteria can be used for terminating the iterations of Algorithm~\ref{alg:update} and is again problem dependent. The simplest would be to impose a maximum rank condition.  That is, stop if $\rank{\widehat \bfZ}$ is greater than some pre-determined integer value $r_{max}.$  However, in many applications a desired rank $r_{max}$ is unknown and alternative stopping criteria should be chosen. One choice is to set a pre-defined average mean square error (aMSE) $f_K$ requiring an estimate of the noise level. The third option is to observe the improvement of $f_K$ in each iteration. Once a rank-$\ell$ update will not substantially improve the aMSE the iteration need to be terminated. A Pareto curve
% 
% (or combination of all three.)
% 
% or improvement.

% If updating the rank of $\widehat \bfZ$ does not improve the solution, thereby implying that a sufficient amount of information for reconstruction has been resolved, then we propose to stop the update process.  Given some user-defined tolerance, $tol$, we propose to stop if 
% \begin{equation}
% 	\label{stopcrit1}
% 	\norm{x_k - x_{k-1}} < tol
% \end{equation} 
% 
% 
% \begin{equation}
% 	\norm{\widehat \bfZ^\ell} < tol
% \end{equation}

\subsubsection*{Using $\widehat \bfZ$ to solve inverse problems.} % (fold)
\label{sub:solving_inverse_problems}
We have described some numerical approaches for computing $\widehat \bfZ$, and it is worthwhile to note that all of these computations can be done off-line.  Once computed, the optimal low-rank regularized inverse matrix $\widehat \bfZ$ can be used to solve inverse problems very quickly, requiring only one matrix-vector multiplication $\widehat\bfZ \bfb$, if $\widehat \bfZ$ is obtained explicitly, or two multiplications $\widehat\bfX (\widehat \bfY\t \bfb)$ if $\widehat \bfX$ and $\widehat \bfY$ are computed separately.
This is particularly important for applications where hundreds of inverse problems need to be solved in real time.  The ability to obtain fast and accurate reconstructions of many inverse problems in real time quickly compensates for the initial cost of constructing $\widehat \bfZ$ (or $\widehat \bfX$ and $\widehat \bfY$).

One of the significant advantages of our approach is that we seek a regularized inverse matrix $\widehat \bfZ$ directly, so that knowledge of the forward model $\bfA$ is not required.  However, for applications where $\bfA$ is desired, our framework could be equivalently used to first estimate a low-rank matrix $\bfA$ from the training data by solving Equation~\eqref{eqn:empiricalmin} where $\bfZ$ is $\bfA,$ and $\bfb^{k}$ and $\bfxi^{k}$ are switched.  This approach is similar to ideas from machine learning \cite{Baldi2001,Hastie2009} and are related to matrix completion problems where the goal is to reconstruct unknown or missing values in the forward matrix \cite{Cai2010}.  Then, once $\widehat \bfA$ is computed, standard regularization methods may be required to solve inverse problems.  If this is the case, we propose to use the training data to obtain optimal regularization parameters, as discussed in \cite{Chung2011}.  Since the ultimate goal in many applications is to compute solutions to inverse problems, obtaining an approximation of $\mxA$ is not necessary and may even introduce additional errors. Thus, if $\mxA$ is not required, our approach to compute the optimal regularized inverse matrix is more efficient and less prone to errors.

\section{Numerical Results} % (fold)
\label{sec:numerical_results}
In this section, we compute optimal low-rank regularized inverse matrices for a variety of problems and evaluate the performance of these matrices when solving inverse problems.  We generated a training set of $K$ observations, $\bfxi^{(k)}$ and $\bfb^{(k)}, k=1,\ldots,K$.  Then for various error measures, $\rho$, we compute optimal low-rank regularized inverse matrices $\widehat \bfZ$ as described in Section~\ref{sec:rank}.  To validate the performance of the computed optimal low-rank regularized inverse matrices, we generate a validation set of $\bar K$ different observations, $\bar \bfb^{(k)}, k=1,\ldots,\bar K,$ and compute the error between the reconstructed solution and the true solution $\bar \bfxi^{(k)}, k=1,\ldots,\bar K$, i.e., $e_k(\widehat \bfZ) = \rho(\widehat \bfZ \bar \bfb^{(k)} - \bar \bfxi^{(k)})$.  For each experiment, we report the sample mean error $f_K(\widehat \bfZ)$ as computed in~\eqref{eqn:samlemeanerror}, along with other descriptive statistics regarding the error distributions.

Standard regularization methods such as TSVD require knowledge of $\bfA$, while our approach requires training data.  Thus, it is difficult to provide a fair comparison of methods.  In our experiments, we compare our reconstructions to two ``standard'' reconstructions:
\begin{itemize}
	\item TSVD-$\bfA$: This refers to the TSVD reconstruction, as computed in \eqref{eqn:TSVD}, and it requires matrix $\bfA$ and its SVD.  The rank-$r$ reconstruction matrix for TSVD-$\bfA$ is given by $\bfZ = \bfA_r^\dagger.$
	\item TSVD-$\widehat \bfA$: If $\bfA$ is not available, one could estimate it from the training data by solving the following rank-constrained optimization problem,
	\begin{equation}
		\label{eqn:getAhat}
		\widehat \bfA = \argmin_{\rank{\bfA} \leq \bar r} \frac{1}{K} \sum_{k = 1}^K \rho\left(\bfA \bfxi^{(k)} - \bfb^{(k)}\right)\,,
	\end{equation}
	for some fixed integer $\bar r.$ We propose to use methods described in Section~\ref{sec:rank} to obtain matrix $\widehat \bfA$.  Then, the computed matrix $\widehat \bfA$ could be used to construct a TSVD reconstruction. We refer to this approach as TSVD-$\widehat \bfA$, and the rank-$r$ reconstruction matrix for TSVD-$\widehat \bfA$ for $r \leq \bar r$ is given by $\bfZ = \widehat\bfA_r^\dagger.$  This seems to be a fair comparison to our approach since only training data is used; however, a potential disadvantage of TSVD-$\widehat \bfA$ is that $\widehat \bfA$ can be large, and computing its SVD may be prohibitive. 
\end{itemize} 

We present results for three experiments: Experiments 1 and 2 are 1D examples, and Experiment 3 is a 2D example.

\subsection{One-dimensional examples} % (fold)
\label{sub:one_dimensional_examples}
For the 1D examples, we use random time signals by generating piecewise constant functions as seen in Figure~\ref{fig:sampleSignal}. 
Here, the number of jumps in each signal is a randomly chosen integer between 3 and 20, drawn from a uniform distribution.  The times for the jumps and the signal strength were both randomly selected in the interval $[0,1]$, also from a uniform distribution.  The discretized signal has length $n = 150.$  The corresponding blurred signals were computed as in Equation~\eqref{eq:inverseproblem}, where matrix $\bfA$ is $150 \times 150$ and represents convolution with a 1D Gaussian kernel with zero mean and variance 2, and $\bfdelta$ represents Gaussian white noise. The noise level for the blurred signals was $0.01$, which means that $\norm[2]{\bfdelta}^2 / \norm[2]{\bfA\bfxi}^2 \approx 0.01$. Training and validation signals were generated in the same manner.  Figure~\ref{fig:sampleSignal} displays four sample signals from the validation set and their corresponding observed signals.  

% Here, the number of jumps are chosen at uniform random between 3 and 20 at uniform random times in the interval $[0,1]$. Signal strength are uniformly chosen in $[0,1]$. The discretized signal is  of length $n = 150.$  The corresponding blurred signals were computed as in Equation~\eqref{eq:inverseproblem}, where matrix $\bfA$ is $150 \times 150$ and represents convolution with a 1D Gaussian kernel with zero mean and variance 2, and $\bfdelta$ represents Gaussian white noise. The noise level for the blurred signals was $0.01$, which means that $\norm[2]{\bfdelta}^2 / \norm[2]{\bfA\bfxi}^2 \approx 0.01$. Training and validation signals were generated in the same manner.  Figure~\ref{fig:sampleSignal} displays sample signals from the validation set and their corresponding blurred signal.  

\begin{figure}[bthp]
\begin{center}
\includegraphics[width=0.9\textwidth]{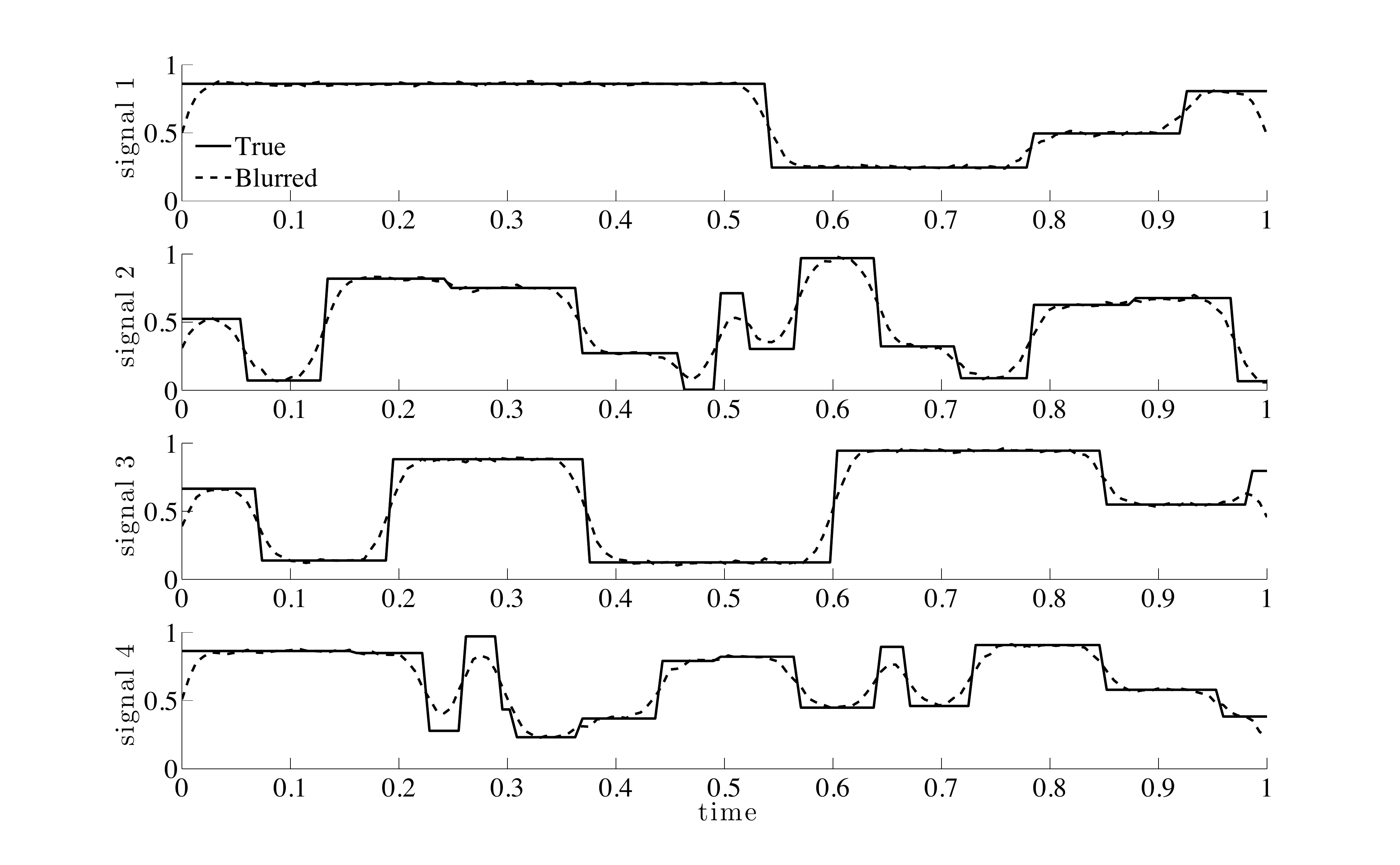}
\end{center}
\caption{Four sample validation signals and corresponding observed signals, used in Experiments 1 and 2.}
\label{fig:sampleSignal}
\end{figure}

%%%%%%%%%%%%%%%%%%%%%%%%%%%%%%%%%%%%%%%%%%%%%%%%%%%%%%%%
\paragraph{Experiment 1.} 
For this experiment, we consider the case where $\rho = \frac{1}{2}\norm[2]{\,\cdot\,}^2$ and compute an \emph{optimal regularized inverse matrix} (ORIM) as described in Section~\ref{sub:2_norm_case}. We refer to this approach as $\rm ORIM_{2}$ and we refer to the computed matrix as $\widehat \bfZ_2$. 

First, we fix the number of training signals to $K= 10,000$.  For varying ranks $r,$ we compute the sample mean error for $\rm ORIM_{2}$, TSVD-$\bfA$, and TSVD-$\widehat \bfA$ for a set of $10,000$ different validation signals.  For computing $\widehat \bfA$, $\bar r$ was fixed to $100$.  Errors plots as a function of rank are shown in Figure~\ref{fig:rankcompareORIM2TSVD_1D}. We observe that the sample mean error for ORIM$_2$ is monotonically decreasing and consistently smaller than the sample mean errors for both TSVD-$\bfA$ and TSVD-$\widehat \bfA$.  This is true for all ranks. Furthermore, both TSVD-$\bfA$ and TSVD-$\widehat \bfA$ exhibit semi-convergence behavior similar to Figure~\ref{fig:tsvd2}, whereby the error initially improves as the rank increases, but reconstructions degrade for larger ranks.  On the other hand, since regularization is built into the design and construction of $\widehat \bfZ_2$, $\rm ORIM_{2}$ avoids semi-convergence behavior altogether, as evident in the flattening of the error curve. Thus, $\rm ORIM_{2}$ solutions are less sensitive to the choice of rank.

% First, we fix the number of training signals $K= 10,000$ and compare the sample mean errors of $\rm ORIM_{2}$ with the sample mean errors of TSVD-$\bfA$ and TSVD-$\widehat \bfA$ for varying ranks $r$ for a set of $10,000$ different validation signals. Results are shown in Figure~\ref{fig:rankcompareORIM2TSVD_1D}. We observe that the sample mean error of ORIM$_2$ is monotonically decreasing with the rank and stays below the sample mean errors of TSVD-$\bfA$ and TSVD-$\widehat \bfA$ for any rank $r$. Furthermore, we find that TSVD-$\bfA$ achieves a lower sample mean error than TSVD-$\widehat \bfA$ for ranks up to $r=57$, however semi-convergence occurs later and slower for TSVD-$\widehat \bfA$. Both TSVD-$\bfA$ and TSVD-$\widehat \bfA$ exhibit semi-convergence behavior similar to Figure~\ref{fig:tsvd2}, whereby the error initially improves as the rank increases, but reconstructions degrade for larger ranks. On the other hand, since regularization is built into the design and construction of $\widehat \bfZ_2$, $\rm ORIM_{2}$ avoids semi-convergence behavior altogether, as evident in the flattening of the error curve. Thus, $\rm ORIM_{2}$ solutions are less sensitive to the choice of rank.

For ranks less than $r=57$, TSVD-$\bfA$ achieves a smaller sample mean error than TSVD-$\widehat \bfA$.  We also note that the smallest error that TSVD-$\widehat \bfA$ can achieve is larger than the smallest error than TSVD-$\bfA$ can achieve.  This is to be expected since $\widehat\bfA$ is an approximation of $\bfA.$  Although the semi-convergence phenomenon for TSVD-$\widehat \bfA$ occurs later and errors grow slower than for TSVD-$\bfA$, both TSVD solutions require the selection of a good choice for the rank.  Since training data is available, one could use the training data to select an ``optimal'' rank for TSVD-$\bfA$ and TSVD-$\widehat \bfA$, i.e., the rank which provides the minimal sample mean error for the training set.  This is equivalent to the opt-TSVD approach from \cite{Chung2011}. For TSVD-$\bfA$ and TSVD-$\widehat \bfA$, the minimal sample mean error for the training set occurs at ranks 51 and 54 (consistent with validation set), respectively.  

\begin{figure}[bthp]
\begin{center}
\includegraphics[width=0.9\textwidth]{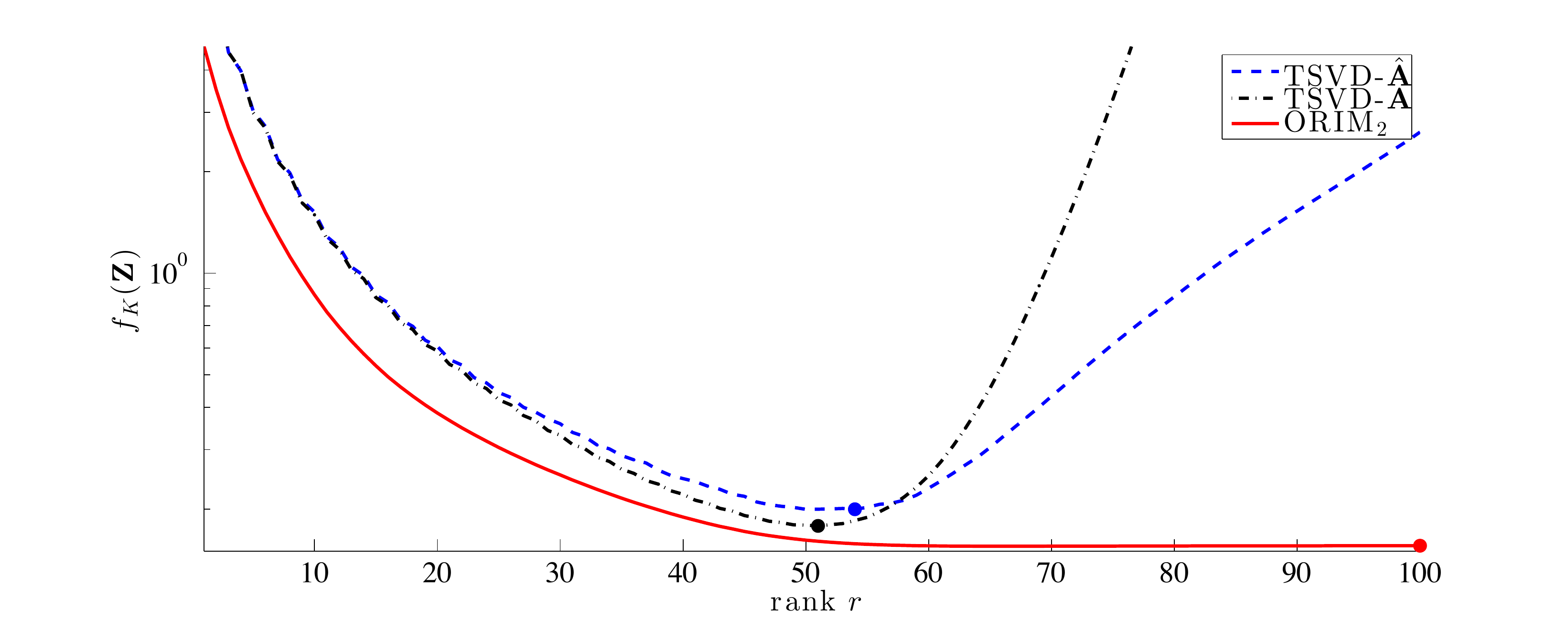}
\end{center}
\caption{Sample mean error $f_K$ on a logarithmic scale for $\rm ORIM_{2}$, TSVD-$\bfA$, and TSVD-$\widehat \bfA$ for the validation signals for various ranks. The dots correspond to the minimal sample mean error for the training data.}
\label{fig:rankcompareORIM2TSVD_1D}
\end{figure}

In Figure~\ref{fig:boxplots1}, we provide box-and-whisker plots to describe the statistics of the sample errors for the validation set, $ e_k(\bfZ) = \frac{1}{2} \norm[2]{\bfZ \bar \bfb^{(k)} - \bar \bfxi^{(k)}}^2$ for $k = 1,\ldots,\bar K$.  Each box provides the median error, along with the 25$^{\rm th}$ and 75$^{\rm th}$ percentiles.  The whiskers include extreme data, and outliers are plotted individually.  For TSVD-$\bfA$ and TSVD-$\widehat \bfA$, errors correspond to the ranks computed using the opt-TSVD approach ($r$ = 51 and 54 respectively), and errors for $\rm ORIM_{2}$ correspond to maximal rank 100.  The average of the errors $e_k(\bfZ)$ correspond to the dots in Figure~\ref{fig:rankcompareORIM2TSVD_1D}.

\begin{figure}[bthp]
\begin{center}
\includegraphics[width=\textwidth]{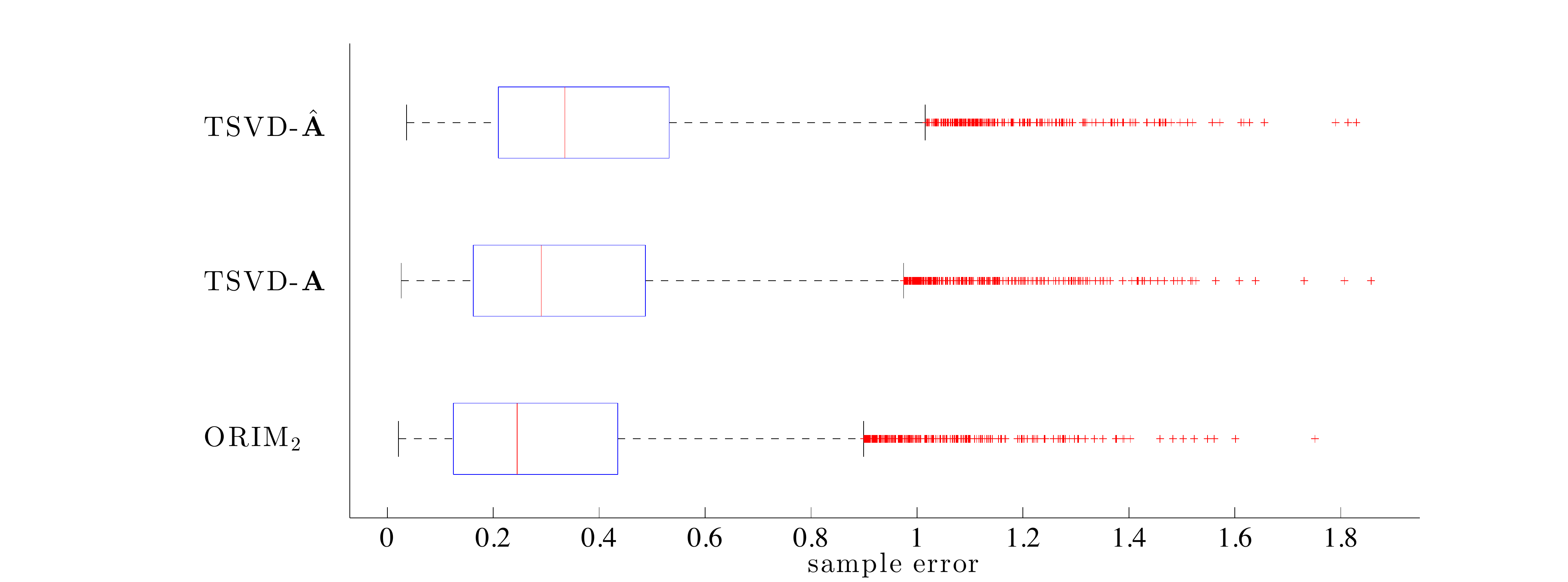}
\end{center}
\caption{Box plots of sample errors for the validation set used in Experiment 1.  All results use the 2-norm.}
\label{fig:boxplots1}
\end{figure}

Next, we investigate how the errors for ORIM$_2$ vary for different numbers of training signals. That is, for increasing numbers of training signals, we compute the $\rm ORIM_{2}$ matrix and record sample errors $e_k(\widehat \bfZ_2),$ for both the training set and the validation set.   It is worthwhile to mention that the rank of $\widehat \bfZ_2$ is inherently limited by the number of training signals and the signal length, i.e., $r \leq \min\{K,n\}$.  For this example, we have $n=150,$ and we select the rank of $\widehat \bfZ_2$ to be
\begin{equation}
	 r = \left\{
  \begin{array}{ll}
    K  & \mbox{for } K < 50,\\
    50 & \mbox{for } K \ge 50\,.
  \end{array}
\right.
\end{equation}
Notice that for rank $r= 50$ with a signal length of $n = 150$, we have a total of $2\cdot r\cdot n = 15,000$ unknowns (for $\bfX$ and $\bfY$).  In Figure~\ref{fig:pareto} we provide the median errors along with the 25$^{\rm th}$ to 75$^{\rm th}$ percentiles.  These plots are often referred to as Pareto curves in the literature.
 
We observe that for the number of training signals $K < 50,$ the ORIM$_2$ matrix is significantly biased towards the training set.  This is evident in the numerically zero computed errors for the training set.  In addition, we expect the errors for the training set to increase and the errors for the validation set to decrease as more training signals are included in the training set, due to a reduction in bias towards the training set. Eventually, errors for both signal sets reach a plateau, where additional training data do not significantly improve the solution. The point at which the errors begin to stabilize indicates a natural choice for the number of training signals to use, e.g., when the relative change in errors reaches a predetermined tolerance. For this example a relative improvement of $10^{-3}$ for the errors for the training set is reached at $1,000$ training images.  

There is an interesting phenomena that occurs in the reconstruction errors for the validation set, whereby the error peaks around $K = 150$.  Although puzzling at first, the peak in errors is intuitive when one considers the problem that is being solved.  
For $K= 150$ matrix $\bfB$ is $150 \times 150.$  If $\bfB$ is full rank ($s = 150$), $\bfP=\bfC$ in Theorem~\ref{thm:Sondermann} and problem~\eqref{eq:min2norm} has the solution,
\begin{align*}
 	\widehat \bfZ_2 &= \argmin_{\rank{\bfZ}\leq r} \frac{1}{K}\norm[\rm F]{\bfZ\bfB -\bfC}^2 \\
&	= \bfC_{r} \bfB^{-1},
 \end{align*} 
which is the special case considered by Eckart-Young and Schmidt \cite{Eckart1936,Schmidt1907}
Notice that $\widehat \bfZ_2$ is uniquely determined by the training set, so reconstructions at this point will be most biased to the training set. 

% rank of $widehat \bfZ$ is limited by the nubmer of training images.  For ranks less than 50 we just compute the higheset rank Z we can adn for ranks greater than 50 we use rank50 matrix.
% For training data up to 50, we recontruct training data extremely well so very biased.  For validation data, we initially see reduction because validation data is similar to training.  However, closer to rank 50, the bias towards the training data maximizes.
% For $K= 150,$ $\bfB\in\bbR^{150 \times 150}$ in Theorem~\ref{thm:Sondermann} is square so $\bfP = \bfC^{150 \times 150}$.  If $\bfB$ is full rank (s = n), then we have $\widehat \bfZ = \bfC_{50} \bfB^{-1}$.  The reconstruction of a training signal would give $\bfC_{50} \bfB^{-1} \bfb = \bfC_{50} \bfi^{(k)}$ where $\bfi^{(k)}$ is the $k$th column of the identity matrix so the $k$th column of $\bfC_{50}$ is $\bfxi^{(k)}$.  Now suppose we have a validation signal $\bar\bfb$ and we reconstruct the validation signal, then $\bfC_{50} \bfB^{-1} \bar\bfb$
% Too biased by the training data, even for less than 150 training
% for number of training images 50, why can we get a good reconstruction of the validation data?

\begin{figure}[bthp]
\begin{center}
\includegraphics[width=\textwidth]{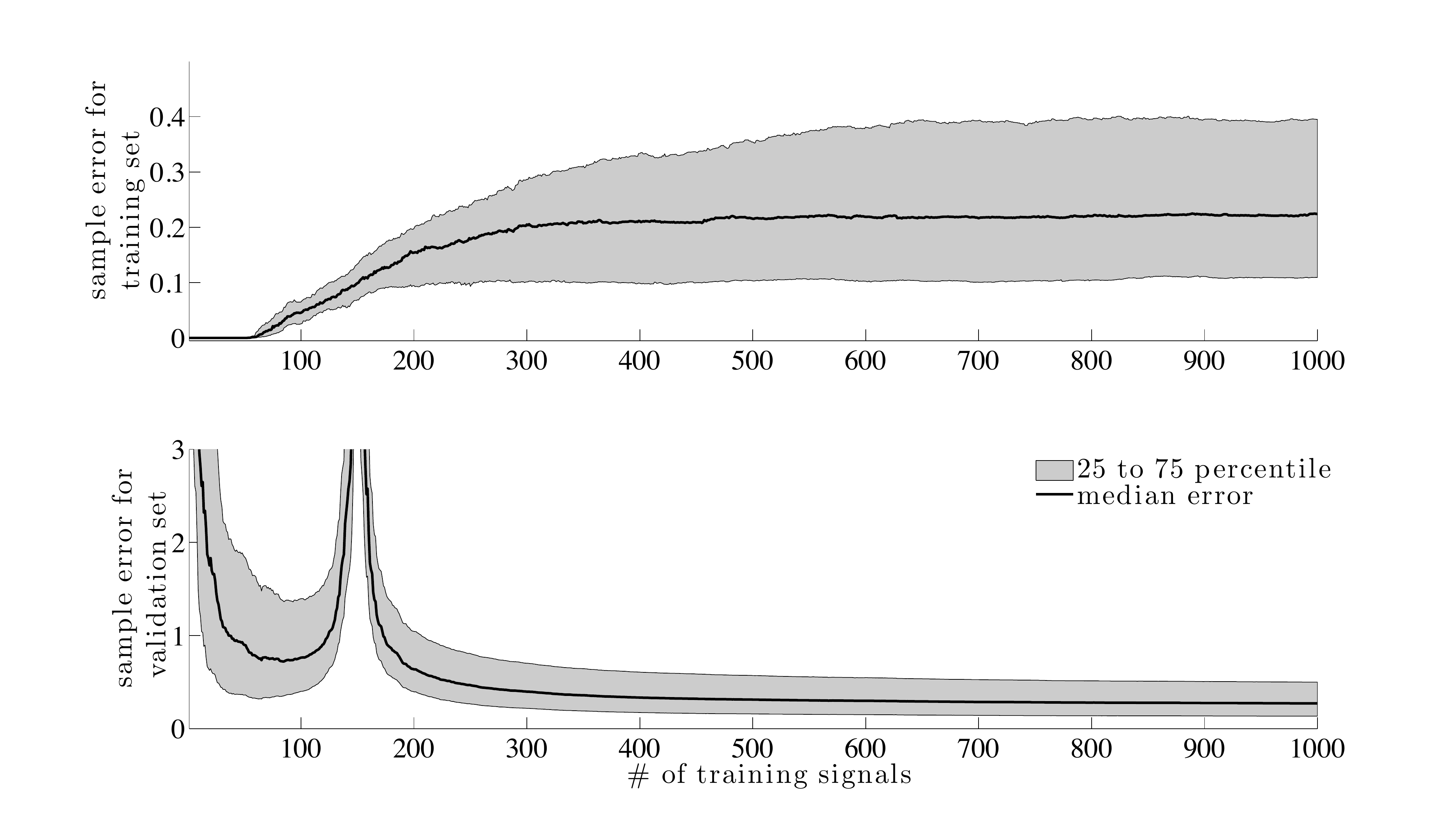}
\end{center}
\caption{Pareto curves.  The curves in the upper figure correspond to sample errors for the training set and the curves in lower figure correspond to sample errors for the validation set.  Here we plot the median values along with the 25$^{\rm th}$ and 75$^{\rm th}$ percentiles of the errors as a function of the number of training signals used to construct $\widehat \bfZ_2$.}
\label{fig:pareto}
\end{figure}

\paragraph{Experiment 2.} 
For this experiment, we use the same setup as described in Experiment~1 for the 1D example, but we compare results obtained by using different error measures $\rho$ in equation~\eqref{eqn:empiricalmin}.  In particular, we consider $\rho = \frac{1}{p}\norm[p]{\,\cdot\,}^p$ for $p=1.2, 2,$ and $5.$  For $p=1.2,$ we use additional smoothing to deal with non-differentiability. We compute an optimal low-rank regularized inverse matrix $\widehat \bfZ_p$ for each $p$.  The computed matrices are then used to reconstruct the validation set.  We refer to these approaches as $\rm ORIM_{1.2}, \rm ORIM_{2}$ and $\rm ORIM_{5}$, respectively.  
Box-and-whisker plots of the sample errors for the validation signals are presented in Figure~\ref{fig:boxplots2}, along with errors for TSVD-$\widehat \bfA$.  In each subfigure, the reconstruction error was computed using a different error measure.  As expected, the reconstructions with smallest error correspond to $\widehat \bfZ_p,$ where $p$ defines the error measure used for the training set.
\begin{figure}[bthp]
\begin{center}
\includegraphics[width=\textwidth]{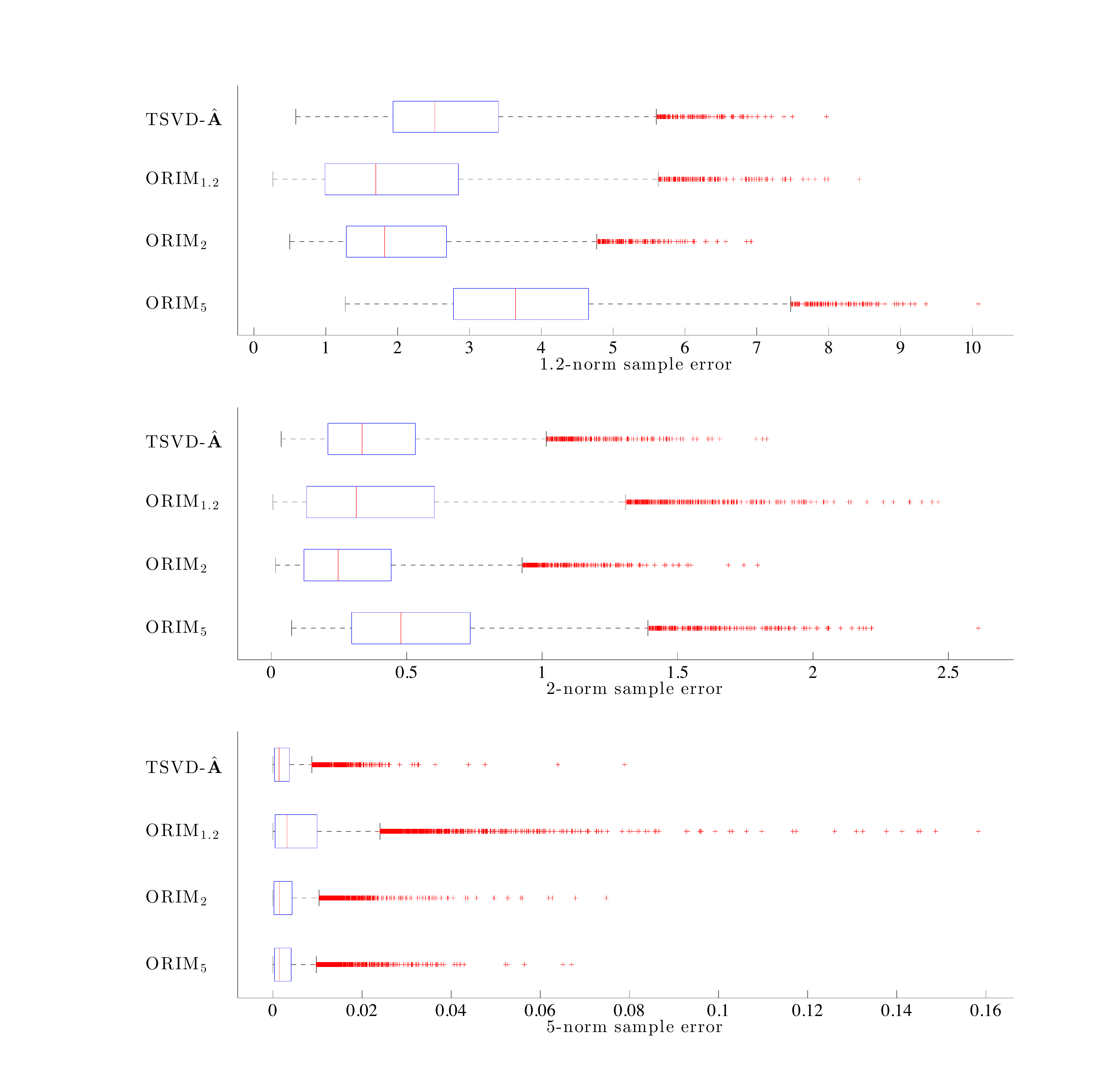}
\end{center}
\caption{Box plots of sample errors for the validation set used in Experiment 2.  Sample errors were computed using $\rho=\frac{1}{p}\norm[p]{\,\cdot\,}^p$ where $p =1.2, 2,$ and $5$, respectively.}
\label{fig:boxplots2}
\end{figure}

\begin{figure}[bthp]
\begin{center}
\includegraphics[width=\textwidth]{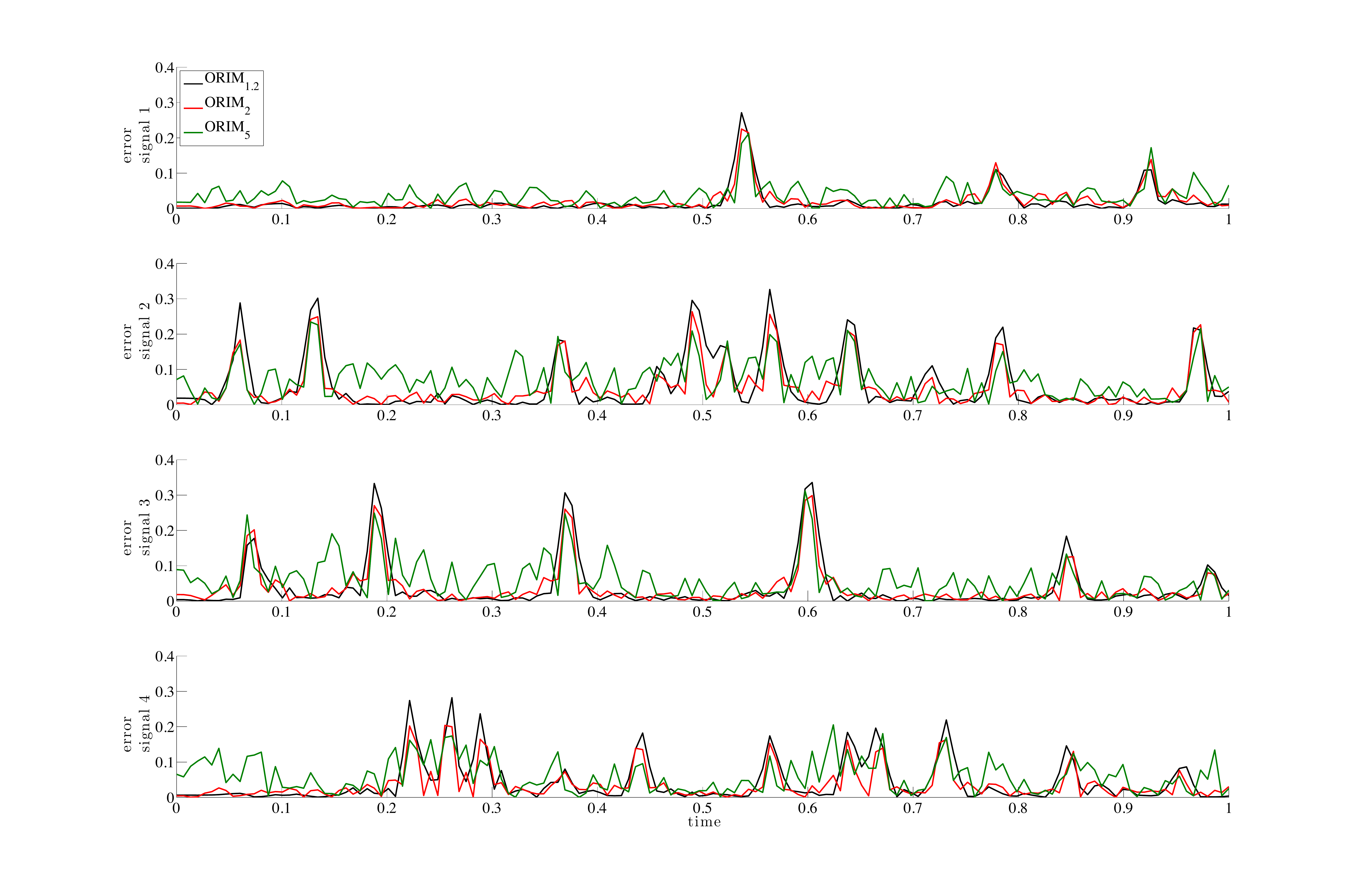} 
\end{center}
\caption{Absolute reconstruction errors for validation signals presented in Figure~\ref{fig:sampleSignal}.}
\label{fig:1Dreconstructions}
\end{figure}
We emphasize that even though the sample errors for the various methods may appear similar in value in the box plots, the reconstructions are significantly different, as seen in the absolute errors.  Plots of the absolute values of the error vectors, $\widehat \bfZ_p \bar\bfb^{(k)} - \bar\bfxi^{(k)},$ for the four validation signals presented in Figure~\ref{fig:sampleSignal} can be found in Figure~\ref{fig:1Dreconstructions}.  We observe that $\rm ORIM_{5}$ allows for large reconstruction errors throughout the signal without major peaks, while absolute reconstruction errors for $\rm ORIM_{1.2}$ tend to remain small in flat regions, with only a few high peaks in the error at locations corresponding to discontinuities in the signal. These observations are consistent with norms in general and illustrate the potential benefits of using a different error measure.

\subsection{Two-dimensional example} % (fold)
\label{sub:two_dimensional_examples}
Next, we describe a problem from image deblurring, where the underlying model can be represented as~\eqref{eq:inverseproblem} where matrix $\bfA$ represents a blurring process, and $\bfxi$ and $\bfb$ represent the true and observed images respectively.  Entries in matrix $\bfA$ depend on the point spread function (PSF).  For image deblurring examples, matrix $\bfA$ is typically very large and sparse.  However, under assumptions of spatial invariance and assumed boundary conditions and/or properties of the PSF, structure in $\bfA$ can be exploited and fast transform algorithms can be used to efficiently diagonalize $\bfA$ \cite{Hansen2006}.  For example, if the PSF satisfies a double symmetry condition and reflexive boundary conditions on the image are assumed, then $\bfA$ can be diagonalized using the 2D discrete cosine transform (DCT).  Of course, these approaches rely on knowledge of the exact PSF.  

Blind deconvolution methods have been developed for problems where the PSF is not known or partially known \cite{Campisi2007}; however, most of these approaches require strict assumptions on the PSF such as spatial invariance or parameterization of the blur function, or they require a good initial estimate of the PSF. 
Our approach uses training data and does not require knowledge of $\bfA$ and, hence, knowledge of the PSF.  Furthermore, we do not require any assumptions regarding spatial invariance of the blur, boundary conditions, or special properties of the blur (e.g., separability, symmetry, etc.).  We refer the reader to \cite{Chung2013a} for comparisons to blind deconvolution methods.

\paragraph{Experiment 3.} 
\begin{figure}[bthp]
\begin{center}
\includegraphics[width=\textwidth]{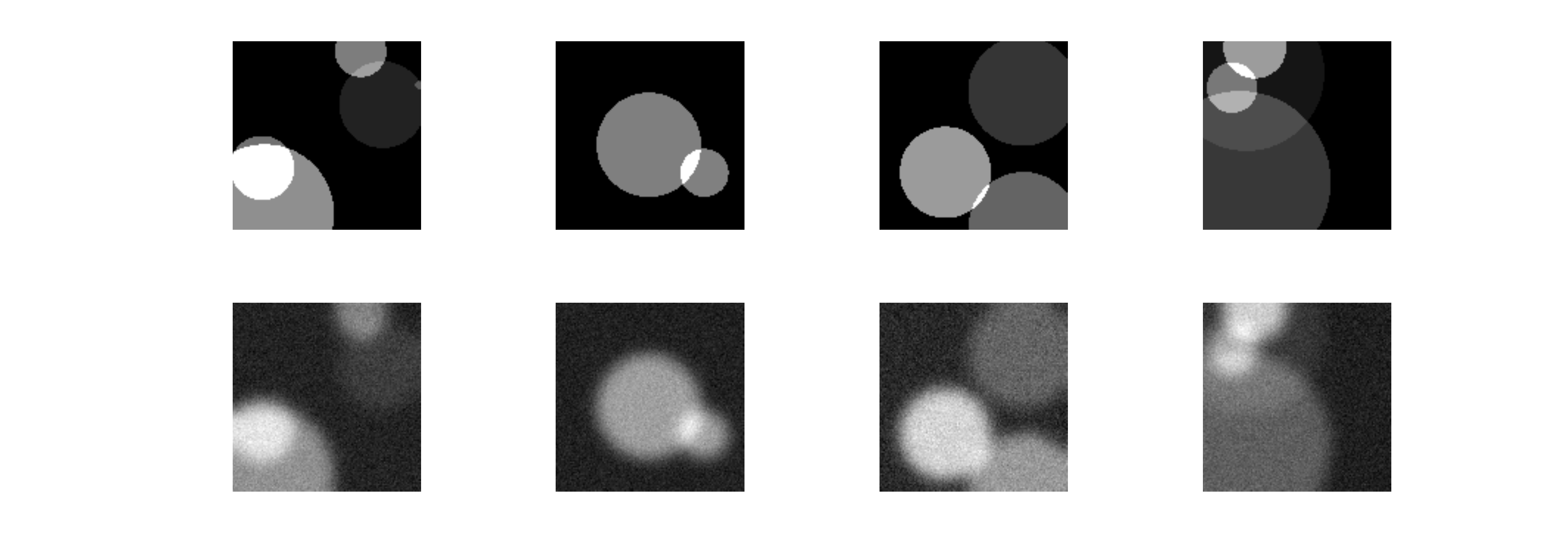}
\end{center}
\caption{Sample true images and corresponding observed, blurred and noisy images that were used for training in Experiment 3.}
\label{fig:training}
\end{figure}
In this experiment, we generate a set of $6,000$ training images.  Each image has $128 \times 128$ pixels and represents a collection of circles, where the number of circles, location, radius and intensity were all randomly drawn from a uniform distribution.  In this example, the PSF was a 2D Gaussian kernel with zero mean and variance 5. We assume reflexive boundary conditions for the image, and Gaussian white noise was added to each image, where the noise level was randomly selected between $0.1$ and $0.15$.  Sample true images along with the corresponding blurred, noisy image can be found in Figure~\ref{fig:training}.  A validation set of $7,000$ images was generated in the same manner.
\begin{figure}[bthp]
\begin{center}
\includegraphics[width=\textwidth]{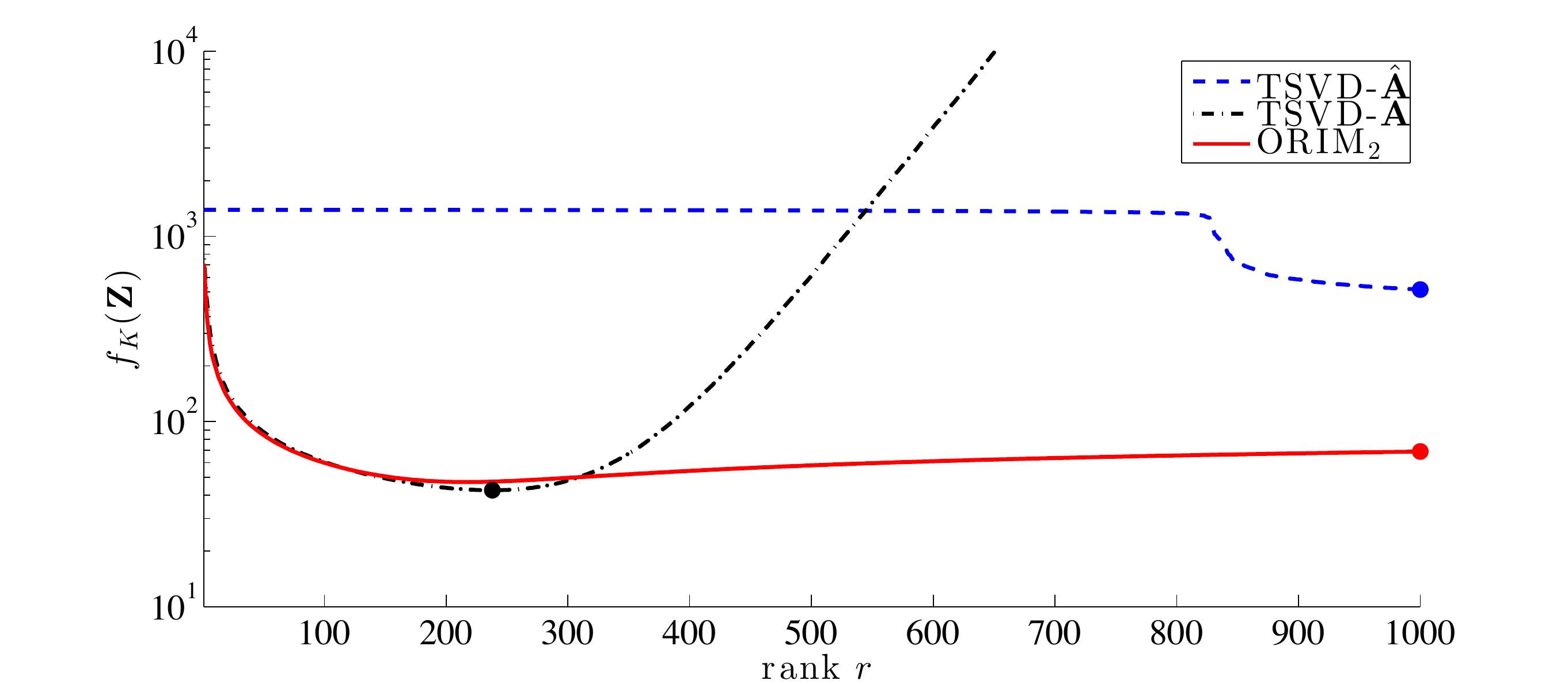}
\end{center}
\caption{Sample mean error for ORIM$_2$, TSVD-$\bfA$, and TSVD-$\widehat \bfA$ in logarithmic scale for validation data in Experiment~3.  The dots correspond to the minimal sample mean error for the training data.}
\label{fig:errors2}
\end{figure}

We compare the sample mean errors for the validation set for $\rm ORIM_{2}$, TSVD-$\bfA$, and TSVD-$\widehat \bfA$, where the maximal rank for $\widehat \bfA$ was $\bar r = 1,000$.  The sample mean errors are provided in Figure~\ref{fig:errors2} for various ranks.  Dots correspond to minimal sample mean error for the training data.  As expected, TSVD-$\bfA$ exhibits clear semi-convergence behavior, where the minimum error is achieved at rank $238$ (for both the training and validation sets).  In contrast, ORIM$_2$ maintains low sample mean errors for various ranks, and only exhibits slight semiconvergence for larger ranks.  This can be attributed to bias towards the training data, since for larger ranks, more training data is required to compensate for more unknowns.  
We observe that the relative difference between the lowest sample mean error for ORIM$_2$ and the lowest sample mean error for TSVD-$\bfA$ is approximately $10\%$.  That is, even without knowledge of $\bfA$, ORIM$_2$ can produce reconstructions that are comparable to TSVD-$\bfA$ reconstructions.

Next, we compare the distribution of the sample errors for ORIM$_2$ and TSVD-$\widehat \bfA$, corresponding to rank $1,000.$  Density plots can be found in Figure~\ref{fig:density2D}.  Among the methods that only require training data, ORIM$_2$ is less sensitive than TSVD-$\widehat \bfA$, implying that it is better to seek a low rank regularized inverse matrix directly, versus seeking a low rank approximation to $\bfA$ and then inverting it.

\begin{figure}[bthp]
\begin{center}
\includegraphics[width=\textwidth]{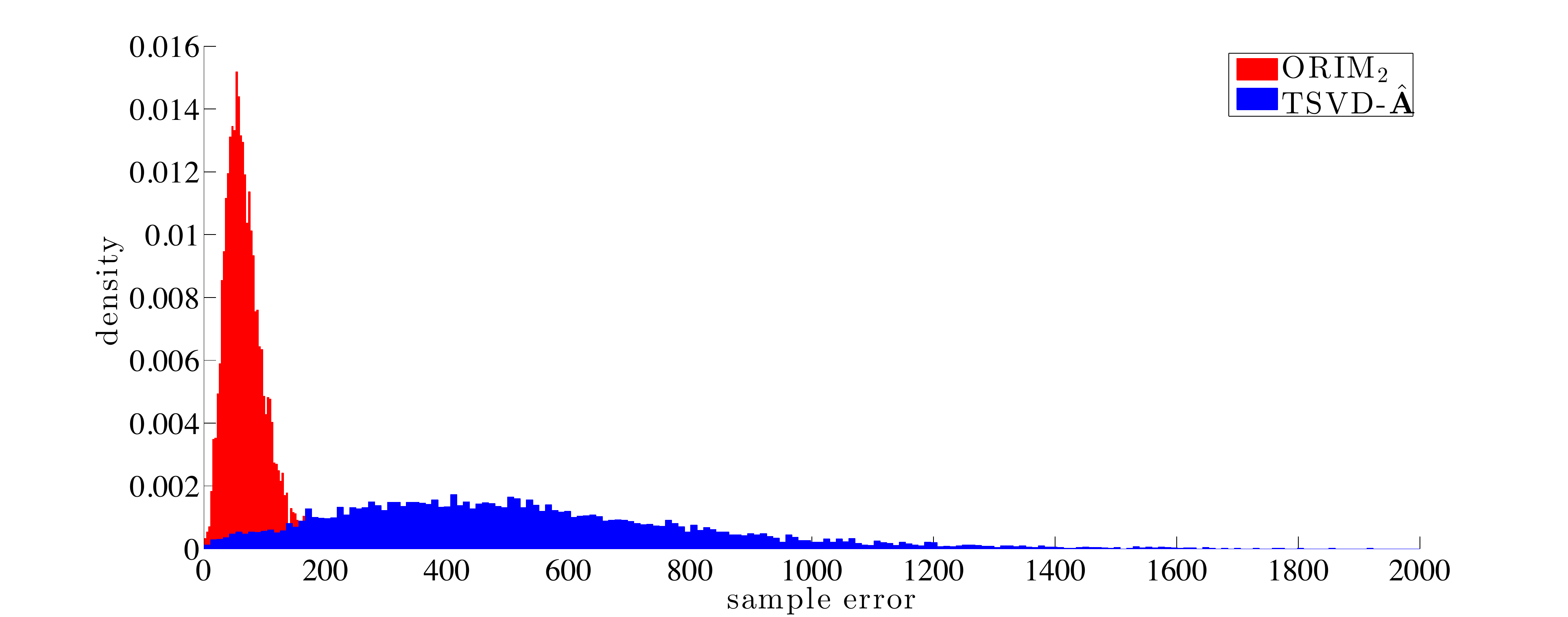}
\end{center}
\caption{Density of sample errors for ORIM$_2$ and TSVD-$\widehat \bfA$ for Experiment 3.}
\label{fig:density2D}
\end{figure}

One of the validation images, with corresponding observed image can be found in the first column of Figure~\ref{fig:reconstructions}.  Reconstructions for ORIM$_2$, TSVD-$\bfA$, and TSVD-$\widehat \bfA$, and corresponding absolute error images between the reconstructed image and the true image (with inverted colormap so that white corresponds to zero error) are provided in Figure~\ref{fig:reconstructions} (top and bottom row respectively).  
\begin{figure}[bthp]
\begin{center}
\includegraphics[width=\textwidth]{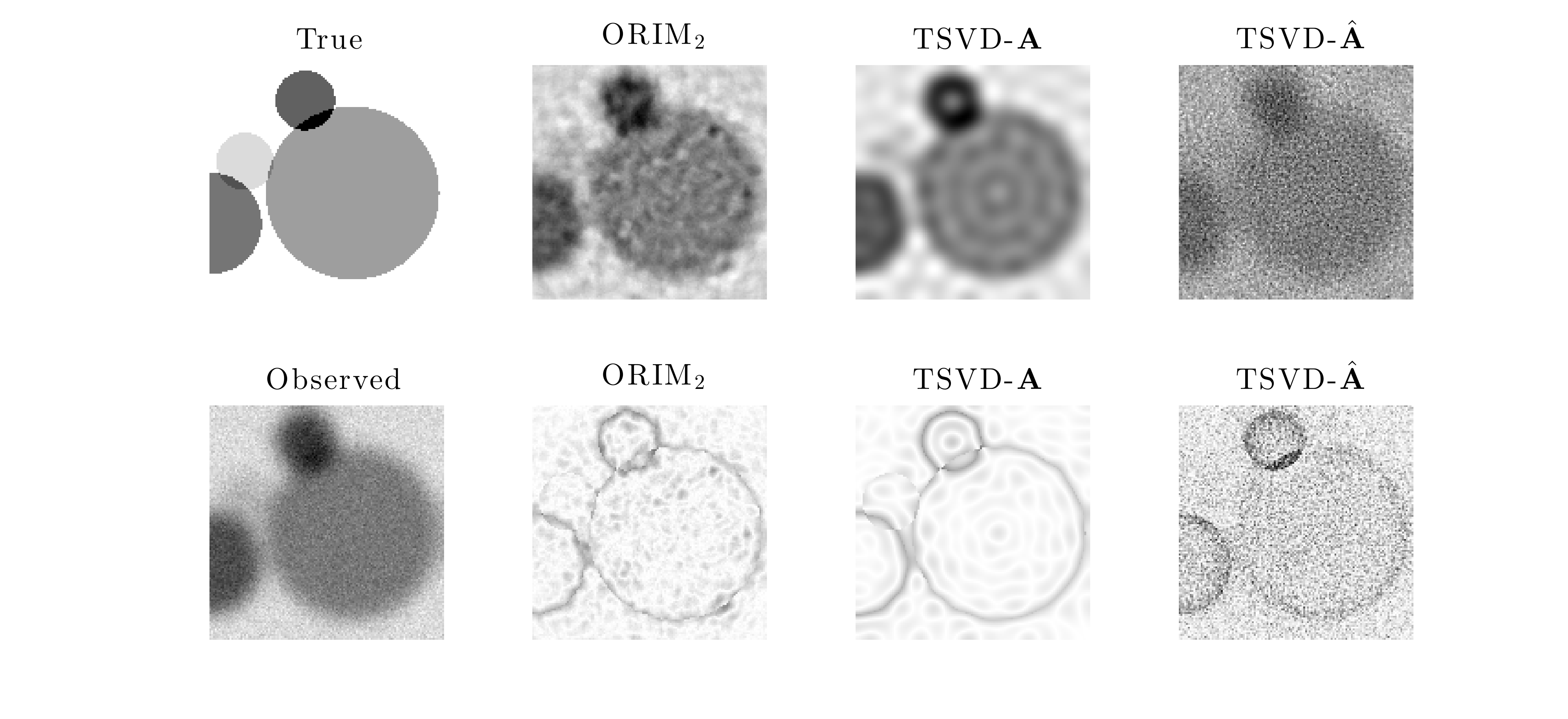}
\end{center}
\caption{Image reconstructions for four validation images for Experiment 3.  The bottom row presents absolute error images between the reconstructed image and the true image.}
\label{fig:reconstructions}
\end{figure}

% section numerics (end)

\section{Conclusions} \label{sec:conclusions} % (fold)
In this paper, we described a new framework for solving inverse problems where the forward model is not known, but training data is readily available.  We considered the problem of finding an optimal low-rank regularized inverse matrix that, once computed, can be used to solve hundreds of inverse problems in real time.  We used a Bayes risk minimization framework to provide underlying theory for the problem.  For the empirical Bayes problem, we proposed a rank-$\ell$ update approach and showed that this approach is optimal in the Bayes framework for the $2$-norm.  Numerical results demonstrate that our approach can produce solutions which are superior to standard methods like TSVD, but without full knowledge of the forward model.  The examples illustrate the benefits of using different error measures and the potential for use in large-scale imaging applications.

% Future work:
% We remark that there are other possibilities for imaging examples such as using an initial approximation and then updating that reconstruction matrix.

% section conclusions (end)

\section*{Acknowledgment}
% \addcontentsline{toc}{section}{Acknowledgment}
We are grateful to Dianne P. O'Leary for helpful discussions.

% \section*{References}

\bibliography{orim}

\end{document}